\documentclass[12pt,oneside]{article}

\usepackage{template}

\title{Stability of parabolic systems of Hodge bundles over punctured $\PP^1$}
\author{Xingyu Cheng}
\date{}

\DeclareMathOperator{\pardeg}{pardeg}

\DeclareMathOperator{\parmu}{par\mu}

\begin{document}

\maketitle

\begin{abstract}
    We consider the problem of existence of semistable systems of Hodge bundles with parabolic structure over a finite set $S \subset \mathbb P^1$ of type $(1,n)$. That is, we consider parabolic Higgs bundles $(\cal E, \theta)$, where $\cal E = \cal L \oplus \cal V$ and $\theta (\cal L) \subset \cal V \otimes \Omega_{\mathbb P^1}^1 (\log S)$, where $\rank \cal L = 1$ and $\rank \cal V = n$. Such systems of Hodge bundles are $\mathbb C^\times$-fixed points in the space of all such (parabolic) Higgs bundles and these correspond to local systems coming from complex variations of Hodge structure under Simpson's correspondence. In the spirit of Agnihotri-Woodward and Belkale, we use enumerative geometry to give numerical criteria for the existence of such semistable parabolic systems of Hodge bundles with semisimple local monodromy. 
\end{abstract}

\tableofcontents

\section{Introduction}

\subsection{The Deligne-Simpson problem}

Given conjugacy classes $C_1, C_2, \ldots, C_s \subset GL(n)$, the Deligne-Simpson problem asks for necessary and sufficient conditions for the existence of matrices $A_i \in C_i$ such that $A_1 \ldots A_s = 1$. This problem is the same asking for the existence of a local system $\cal L$ over $\PP^1 \sm S$ for $S = \set{p_1, \ldots, p_s}$ where the $A_i$ are viewed as the local monodromy of $\cal L$. Such local systems are equivalent to representations $\rho :\pi_1 (\PP^1 \sm S, b) \to GL(\cal L_b)$ for some base point $b \in \PP^1 \sm S$. The fundamental group is presented as $\pi_1 (\PP^1 \sm S, b) = \braket{\gamma_1, \gamma_2, \ldots, \gamma_s \mid \gamma_1 \gamma_2 \ldots \gamma_s = 1}$, where $\gamma_i$ denotes the class of a curve $\gamma_i$ that goes around the puncture $p_i$ exactly once (see Figure \ref{fig:monodromy}). Here, we can interpret $A_i$ as the image of $\gamma_i$ under $\rho$, which is the same as the local monodromy of the local system $\cal L$. 

\begin{figure}[ht]
    \centering
    \begin{tikzpicture}
        \filldraw (0,0) circle (2pt);
        \draw (2,1) circle (2pt) node[anchor=south] {$x_1$}
        (0,2) circle (2pt) node[anchor=south] {$x_2$}
        (-2,1) circle (2pt) node[anchor=south] {$x_3$}
        (1,-1.5) circle (2pt) node[anchor=south] {$x_n$};
        \draw (-1,-1.5) node {$\ldots$};
        \draw[->] (2.5,1) arc (0:360:.5);
        \draw (2.5,1) node[anchor=west] {$\gamma_1$};
        \draw[->] (.5,2) arc (0:360:.5);
        \draw (.5,2) node[anchor=west] {$\gamma_2$};
        \draw[->] (-1.5,1) arc (0:360:.5);
        \draw (-1.5,1) node[anchor=west] {$\gamma_3$};
        \draw[->] (1.5,-1.5) arc (0:360:.5);
        \draw (1.5,-1.5) node[anchor=west] {$\gamma_n$};
        \draw (0,0)--(1.55,.8);
        \draw (0,0)--(0,1.5);
        \draw (0,0)--(-1.55,.8);
        \draw (0,0)--(.65,-1.1);
        \draw (-.1,0) node[anchor=north] {$b$};
    \end{tikzpicture}
    \caption{Picture of monodromy}
    \label{fig:monodromy}
\end{figure}
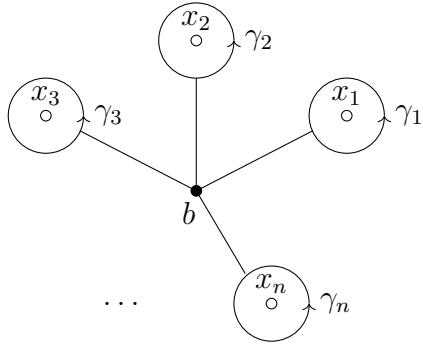

\subsection{The unitary existence problem}

Suppose now that each matrix $A_i$ is unitary with eigenvalues 
$$
e^{2 \pi i \alpha_1 (p_i)}, \ldots e^{2 \pi i \alpha_n (p_i)}
$$ 
with $0 \leq \alpha_1 (p_i) \leq \alpha_2 (p_i) \leq \ldots \leq \alpha_n (p_i) < 1$. Mehta-Seshadri \cite{mehta1980moduli} showed that irreducible local systems with (special) unitary local monodromy matrices and prescribed eigenvalues as in above are in correspondence with stable parabolic vector bundles of parabolic degree 0, which are vector bundles $\cal V$ such that for each $p \in S$, $\cal V_p$ is equipped with a (complete) flag and to each subspace of the flag is attached parabolic weights given by the $\alpha_j(p)$'s (see Definition \ref{def: parabolic bun}). In this case, the problem of existence of unitary matrices in prescribed conjugacy classes multiplying to one is equivalent to the existence of (semi)stable parabolic bundles of degree 0.

\subsubsection{Preliminary notation}
We will denote by $Gr(r,W)$ the Grassmannian of dimension $r$ subspaces of a vector space $W$. We will also denote $Gr(r,\C^n) = Gr(r,n)$. 
\begin{defn*}
Let $F_\bullet : F_1 \subset F_2 \subset \ldots \subset F_n = W$ be a complete flag of subspaces of an $n$-dimensional vector space $W$. For a subset $I = \set{i_1 < i_2 < \ldots < i_r} \subset [n] = \set{1,2, \ldots, n}$ of cardinality $r$, the \textbf{Schubert variety} $\Omega_I (F_\bullet) \subset Gr(r,W)$ is
$$
\Omega_I (F_\bullet) = \set{U \in Gr(r,W) \mid \dim (U \cap F_{i_k}) \geq j, j = 1, \ldots, r}. 
$$
We will denote by $\sigma_I \in A^* Gr(r,W)$ the cycle class of $\Omega_I (F_\bullet)$. We will also write 
$$
\codim \sigma_I = \codim \Omega_I(F_\bullet) = \sum_{j=1}^r (n-r+j - i_j). 
$$
\end{defn*}

\begin{defn*}
    Let $I^{p_1}, \ldots, I^{p_s} \subset [n]$ be subsets each of cardinality $r$ and $d \geq 0$ an integer. The \textbf{Gromov-Witten} number $\braket{\sigma_{I^{p_1}}, \ldots, \sigma_{I^{p_s}}}_d$ is the number of stable maps (counted as zero if infinite) $f: \PP^1 \to Gr(r,n)$ of degree $d$ such that $f(p_i) \in \Omega_{I^{p_i}} (F_\bullet (p_i))$ for complete flags $F_\bullet(p_1), \ldots, F_\bullet (p_s)$ of $W$ in general position.  
\end{defn*}

Such Gromov-Witten numbers can be computed using the quantum cohomology of Grassmannians \cite{BERTRAM1997quantumschubert}.

\subsubsection{Existence of semistable parabolic bundles}

Because (semi)stability is an open condition, it suffices to look at generic parabolic bundles of parabolic degree 0. We can then apply the shifting operation described in Section \ref{section:shifting} to get the sum of weights
$$
\sum_{p \in S} \sum_{i=1}^n \alpha_i (p) = 0.
$$
Doing this forces the parabolic bundle to have degree 0. Generic bundles of rank $n$ and degree $0$ are isomorphic to the trivial bundle $\cal O^{\oplus n}$. This then reduces the problem to checking the semistability of the trivial bundle $\cal O^{\oplus n}$ with generic flags and prescribed parabolic weights. 

For such a parabolic bundle to be semistable, we need every subbundle $\cal V$, with parabolic structure induced by how it intersects the data of flags, to satisfy the parabolic slope inequality $\parmu \cal V \leq \parmu \cal O^{\oplus n} = 0$ (see Definition \ref{def: parabolic bun}), that is, we must not have a subbundle that contradicts this inequality. A rank $r$ subbundle $\cal V \subset \cal O^{\oplus n}$ have a degree, $-d$, and induced weights determined by how $\cal V$ intersects the parabolic flags of $\cal O^{\oplus n}$. How $\cal V$ intersects the flag at $p$ can be recorded by which Schubert cell $\Omega_I (F_\bullet (p)) \subset Gr(r, n) = Gr(r, \cal O^{\oplus n}_p)$ contains $\cal V_p$. The parabolic slope of $\cal V$ is completely determined by this data of degree and weights, and so the question of which inequalities to write becomes one of existence of subbundles of $\cal V \subset \cal O^{\oplus n}$ with prescribed degree such that $\cal V_p$ lies in some Schubert variety in $Gr(r,n) = Gr(r, \cal O^{\oplus n}_p)$ for $p \in S$. 

The question of which inequalities to write is then an enumerative problem. Degree $-d$ subbundles of $\cal O^{\oplus n}$ are exactly in correspondence with degree $d$ maps $f : \PP^1 \to Gr(r,n)$ and to ask for $\cal V_p$ to lie in a Schubert variety $\Omega_{I^p} (F_\bullet (p))$ is exactly the same as as having $f(p) \in \Omega_{I^p} (F_\bullet (p))$. The Gromov-Witten invariant $\braket{\sigma_{I^{p_1}}, \ldots, \sigma_{I^{p_s}}}_d$ records exactly this information. 

Agnihotri-Woodward \cite{agnihotri1998eigenvalues} and Belkale \cite{belkale2001local} were then able to give numerical conditions for the existence of such semistable parabolic bundles in terms of nonvanishing Gromov-Witten numbers and semistability inequalities.
\begin{thm*}[The unitary existence theorem]
    A semistable parabolic bundle of parabolic degree zero and weights $\alpha_1 (p) \leq \alpha_2 (p) \leq \ldots \leq \alpha_n (p) < \alpha_1 (p) + 1$ for each $p \in S$ exists if and only if the following condition holds:
    \begin{itemize}
        \item For any integers $r, d$ with $0 < r < n$, and subsets $I^p = \set{i^p_1 < i^p_2 < \ldots < i^p_r}$ such that $\braket{\sigma_{I^{p_1}}, \ldots, \sigma_{I^{p_s}}}_d \neq 0$, the semistability inequality
        $$
        -d + \sum_{p \in S} \sum_{i \in I^p} \alpha_{n-i+1} (p) \leq 0
        $$
        must hold.
    \end{itemize}
\end{thm*}

\subsection{The existence problem for variation of Hodge structures}

On the other hand, Simpson (\cite{simpson1990harmonic}, \cite{simpson1992higgs}, \cite{simpson1994moduli}) showed that $GL(n)$ solutions of such matrix problems can be associated with certain parabolic Higgs bundles (see Section \ref{section: par higgs}). In particular, there is a correspondence between filtered local systems on one side and parabolic Higgs bundles on the other. In the space of such (parabolic) Higgs bundles, there are certain special $\C^\times$ fixed points called systems of Hodge bundles (see \cite{simpson1990harmonic}), which are in association with complex variations of Hodge structures. Such local systems are interesting, for example, because these are the local systems that come from geometry and all rigid local systems over a multiply punctured $\PP^1$ in the sense of \cite{katz1996rigid} are of this type.

We can set up a general algebro-geometric problem. Let $(\cal E, \theta)$ be a parabolic system of Hodge bundles. Then by Definition \ref{def: shb}, we must have that 
$$
\cal E = \cal E_1 \oplus \cal E_2 \oplus \ldots \oplus \cal E_N
$$
and $\theta$ must satisfy $\theta |_{\cal E_i} : \cal E_i \to \cal E_{i+1} \otimes \Omega^1_{\PP^1} (\log S)$. In this parabolic system of Hodge bundles case, we can assume that the parabolic flags are compatible with this decomposition of $\cal E$ \cite{simpson1992higgs}, i.e. the subspaces making up the flag of $\cal E_p$ can be broken up into a direct sum of subspaces of $\cal E_{i,p}$. We can then give each $\cal E_i$ the data of flags with weights at each point $p \in S$. Set $\alpha^i_1 (p) \leq \alpha^i_2 (p) \leq \ldots \leq \alpha^i_{r_i} (p)$ be the parabolic weights of $\cal E_{i,p}$, and denote the flags of $\cal E_{i,p}$ as $F_1^i(p) \subset F_2^i (p) \subset \ldots \subset F_{r_i}^i (p)$, where the subspace $F^i_j(p)$ will be associated with weight $\alpha^i_{n-i+1} (p)$.

\begin{question*}
\label{main question}
If we fix the data of $\rank \cal E_i$, $\deg \cal E_i$, $\rank \theta |_{\cal E_i}$, and weights $\set{\alpha^i_j (p) \mid 1 \leq j \leq r_i, i \in [N], p\in S}$, then does there exist a semistable parabolic system of Hodge bundles $(\cal E,\theta)$, with nonzero $\theta$, where $\cal E_i$ has flags with weights $\alpha^i_j (p)$ for each $j$ and $p$ as in the above?
\end{question*}

The problem may look very different for different numerical data, and so this question actually consists of many different problems, each perhaps requiring separate techniques. Another question to ask is to ask exactly what numerical data is needed to isolate an irreducible component of the moduli space of all such objects, and therefore make the question more tractable for enumerative techniques. We explain more about the general problem in Section \ref{section:general ag prob} and give some examples of some situations that give rise to irreducible components.

The difference between our situation of systems of Hodge bundles and the unitary case is that the existence of our Higgs field $\theta$ imposes certain restraints. In the unitary case, we can assume our parabolic bundle are as general as possible and that our flags are as general as possible. However the condition that $\theta : \cal E_i \to \cal E_{i+1} \otimes \Omega^1 (\log S)$ along with the numerical data of $\rank \theta \mid_{\cal E_i}$ imposes restrictions on the bundles $\cal E_i$. Furthermore, $\res_p \theta$ (see Definition \ref{def: shb}) is required to preserve the flags of the parabolic systems of Hodge bundles, i.e. $\res_p \theta (\cal E_{\alpha(p)}) \subset \cal E_{\alpha(p)} \otimes \Omega^1 (\log S)_p$, where $\cal E_{\alpha(p)}$ denotes the (largest) part of the flag of $\cal E_p$ with weight $\alpha(p)$. This then imposes additional restraints on the parabolic flags and hence we cannot assume the flags to be generic in this case. 

\subsection{The \texorpdfstring{$(1,n)$}{(1,n)} problem}
We will not attempt to solve the question for every possible combination of numerical data. We restrict our attention to solving the case where where our bundle is of type $(1,n)$ (explained below) with distinct parabolic weights, along with the $(1,2)$ case when we ask only for semisimple local monodromies instead of distinct weights. This is the main goal of this present work.

Restricting to the $(1,n)$ case means $\cal E$ only breaks up into two bundles $\cal E_1 \oplus \cal E_2$ where $\cal E_1 = \cal L$ is a line bundle and $\cal E_2 = \cal V$ is rank $n$. For $p \in S$, we give $\cal L$ weights $\alpha(p)$ and $\cal V$ weights $\beta_1 (p) < \beta_2 (p) < \ldots < \beta_n(p)$. We further ask that $\alpha(p) \neq \beta_i (p)$ for all $i$ as well, so the all the weights are distinct or generic. Furthermore, through parabolic shifting (see Section \ref{section:shifting}), we can further assume that $\alpha(p) < \beta_1 (p)$ for all $p$.

Semistability of such $(\cal E, \theta)$ are determined by $\theta$-stable subbundles of $\cal E$, and so in this $(1,n)$ situation, there are two types of subbundles to consider:
\begin{itemize}
    \item[\ref{sub of V}] All subbundles $\cal S \subset \cal V$. 
    
    \item[\ref{containing L}] Subbundles of the form $\cal L \oplus \cal S$ where $\cal S \subset \cal V$ is a subbundle such that $\cal S \otimes \Omega^1 (\log S)$ contains the image $\theta(\cal L)$.
\end{itemize}
We must handle these two situations separately and they give rise to two different classes of slope inequalities. 

We must also consider two different types of bundles as well. Let $\cal W = \cal V \otimes \Omega^1 (\log S)$. Set $w = \deg \cal W$ and $d = \deg \cal L$. If $d$ is too large, then we cannot assume $\cal W$ and hence $\cal V$ to be generic, and so this case must be dealt with separately. 
\begin{enumerate}[]
    \item[\ref{small sub}] If $d \leq \ceil{w/n}$, then in this case we can deform $\cal L \subset \cal W$ so that $\cal W$ is a generic degree $w$ bundle containing $\cal L$. 

    \item[\ref{large sub}] If $d > \ceil{w/n}$, say $\cal L \simeq \cal O(d)$, then the most generic $\cal W$ containing $\cal L$ we can deform to is $\cal W \simeq \calO (d) \oplus \cal G_{-d, n-1}$, where $\cal G_{-d, n-1}$ denotes a generic degree $-d$ rank $n-1$ bundle. A generic rank $n$ bundle over $\PP^1$ is one of the form $\bigoplus_{i=1}^n \cal O(e_i)$ where $|e_i - e_j| \leq 1$ for all $i, j$.
\end{enumerate}
From hereon, we will always assume $\cal V$, equivalently $\cal W$, to be deformed to be as generic as possible as described above.

\subsubsection{Generalized Gromov-Witten numbers}
\label{section:generalized GW}
A main ingredient in the $(1,n)$ case will the generalized Gromov-Witten numbers, which will allow us to count subbundles beyond just the trivial bundle. We can replace $\cal O^{\oplus n}$ by a generic degree $-D$ vector bundle and play the same game with counting subbundles. A generic rank $n$ vector bundle over $\PP^1$ takes the form $\cal W \simeq \bigoplus_{i = 1}^n \cal O(e_i)$, where $|e_i - e_j| \leq 1$ for all $i, j$. In this situation, we will also call generic vector bundles \textbf{evenly split}. Evenly split bundles are unique in the sense that for any fixed degree $-D$ and rank $n$, there is a collection of integers $(e_1, \ldots, e_n)$ such that any evenly split bundle of this degree and rank is isomorphic to $\bigoplus_{i=1}^n \cal O(e_i)$. This leads to the notion of generalized Gromov-Witten invariants (see, for example, Section 2.4 of \cite{belkale2008quantum} and Section 4.3 of \cite{belkale2022rigid}).

\begin{defn}
    Let $I^1, \ldots, I^s$ be subsets of $[n]$ with cardinality $r$ each. Let $d \in \Z$. Let $\cal W$ be a generic vector bundle over $\PP^1$ of degree $-D$, and suppose we have general (increasing) flags $F_\bullet (p) \in \op{Fl}(\cal W_p)$ for all $p \in S$. 

    The generalized Gromov-Witten number $\braket{\sigma_{I^1}, \ldots, \sigma_{I^s}}_{d, D}$ is defined the be number of degree $-d$ rank $r$ subbundles (counted as 0 if infinite) $\cal V \subset \cal W$ such that for $p = p_j \in S$, $\cal V_p \in \Omega_{I^j} (F_\bullet (p) ) \subset Gr(r, \cal W_p)$. 
\end{defn}
One can reduce the calculation of these generalized Gromov-Witten numbers to the regular Gromov-Witten numbers via the parabolic shifting operation described in Section \ref{section:shifting}. If we pick a point $p = p_j \in S$ to apply the parabolic shift operation at, we get that
$$
\braket{\sigma_{I^1}, \ldots, \sigma_{I^s}}_{d, D} = \braket{\sigma_{K^1}, \ldots, \sigma_{K^s}}_{d', D-1}
$$
where
\begin{itemize}
    \item $d' = d$ if $1 \notin I^j$ and $d' = d - 1$ if $1 \in I^j$;
    \item $K^k = I^k$ for $k \neq j$;
    \item $K^j = \set{i^j_1 - 1, i^j_2 - 1, \ldots, i^j_r - 1}$ if $1 \notin I^j$ and $K^j = \set{i^j_2 - 1, \ldots, i^j_r - 1, n}$ if $1 \in I^j$.
\end{itemize}

\subsubsection{The main \texorpdfstring{$(1,n)$}{(1,n)} theorem}
Recall that we have made the assumption that $\alpha(p) < \beta_1 (p) < \beta_2 (p) < \ldots < \beta_n(p)$. The data that we need to fix for this problem is the data of $\deg \cal L$, $\deg \cal V$, and the data of parabolic weights. If we have fixed all such data, then by definition
$$
\pardeg \cal E = \deg \cal L + \deg \cal V + \sum_{p \in S} \paren*{\alpha(p) + \sum_{i=1}^n \beta_i (p)}
$$
is a fixed number. In the statement of the theorem, we also have a corrective factor $r(2-s)$. This corrective factor comes from us looking at subbundles of $\cal W = \cal V \otimes \Omega^1 (\log S)$, and so to get a subbundle of $\cal V$, we need to twist by $\paren*{\Omega^1 (\log S)}^\vee \simeq \cal O(-2 + s)^\vee \simeq \cal O(2-s)$. Thus for a rank $r$ subbundle $\cal S \subset \cal W$, to get a subbundle $\cal V$, we need $\cal S \otimes \cal O(2-s)$, which has degree $\deg \cal S + r(2-s)$.

\begin{thm}[Existence of semistable $(1,n)$ parabolic system of Hodge bundles with distinct weights]
\label{thm:(1,n)}
A semistable system of Hodge bundles of type $(1,n)$ with weights $\alpha (p)$ and $\beta_1 (p) < \beta_2 (p) < \ldots < \beta_n (p)$ as above with $\deg \cal L$ and $\deg \cal V$ fixed exist if and only if the following conditions in the different cases of bundles are satisfied. 

\begin{enumerate}
        \item[\ref{small sub}] If $\deg \cal L \leq \ceil{w/n}$, then the following two types of numerical conditions must be met.
        
        \begin{itemize}
        \item[\ref{sub of V}] For every $I^{p_1}, \ldots, I^{p_s} \subset [n]$ of size $r$ such that the usual Gromov-Witten number 
        $$
        \braket{\sigma_{I^{p_1}}, \ldots, \sigma_{I^{p_s}}}_\delta \neq 0,
        $$
        the inequality
        $$
        \frac{-\delta + r(2 - s) + \sum_{p \in S} \sum_{i \in I^p} \beta_{n-i+1} (p)}{r} \leq \frac{\pardeg \cal E}{n+1}
        $$
        holds.

        \item[\ref{containing L}] For every $J^{p_1}, \ldots, J^{p_s} \subset [n-1]$ of size $r-1$ such that the generalized Gromov-Witten number $\braket{\sigma_{J^{p_1}}, \ldots,  \sigma_{J^{p_s}}}_{\delta, d} \neq 0$, the inequality
        $$
        \frac{2d - \delta + r(2 - s) + \sum_{p \in S} \paren*{\alpha(p) + \beta_1(p) + \sum_{j \in J^p} \beta_{n-j+1} (p)}}{r+1} \leq \frac{\pardeg \cal E}{n+1}
        $$
        holds.
        \end{itemize}

        \item[\ref{large sub}] If $\deg \cal L > \ceil{w/n}$, then the following two types of numerical conditions must be met.
        
        \begin{itemize}
        \item[\ref{sub of V}] For every $I^{p_1}, \ldots, I^{p_s} \subset [n]$ of size $r$ such that $\braket{\sigma_{I^{p_1}}, \ldots, \sigma_{I^{p_s}}}_{\delta, -w} \neq 0$, the inequality
        $$
        \frac{- \delta + r(2-s) + \sum_{p \in S} \sum_{i \in I^p} \beta_{n-i+1} (p)}{r} \leq \frac{\pardeg \cal E}{n+1}
        $$
        is satisfied.
        
        \item[\ref{containing L}] For every $J^{p_1}, \ldots, J^{p_s} \subset [n-1]$ of size $r-1$ such that $\braket{\sigma_{J^{p_1}}, \ldots, \sigma_{J^{p_s}}}_{\delta, d - w} \neq 0$, then the inequality 
        $$
        \frac{2d - \delta + r(2-s) + \sum_{p \in S} \paren*{\alpha(p) + \beta_1(p) + \sum_{j \in J^p} \beta_{n-j+1}(p)}}{r+1} \leq \frac{\pardeg \cal E}{n+1}
        $$
        is satisfied.
    \end{itemize}
    \end{enumerate}
\end{thm}

If $n=2$, we could in fact strengthen the conditions on the weights. Here we can allow the weights $\alpha(p), \beta_1 (p), \beta_2(p)$ to be equal. We can then apply the parabolic shifting operation at every point of $S$ to get rid of the parabolic structure and get similar inequality conditions for existence of semistable $(1,2)$ bundles. We explore this in Section \ref{section: (1,2)} and in particular Theorem \ref{thm:(1,2)} is another one of our main results.

We also analyze in Section \ref{section: (1,1)} parabolic systems of Hodge bundles of type $(1,1)$. Such local systems correspond to $U(1,1)$ representations of $\pi_1 (\PP^1 \sm S)$ and so must not be $U(2)$ local systems. We give a proof that stable $(1,1)$ parabolic systems of Hodge bundles cannot satisfy the criterion given in \cite{Biswas1998ACF} for the existence of $U(2)$ local systems. The $(1,1,\ldots,1)$ case was also analyzed with particular focus on $(1,1,1)$ systems of Hodge bundles in Section \ref{section: (1,...,1)}.

\subsection{Acknowledgment}
I wish to thank my advisor Prakash Belkale for many helpful discussions.

\section{Basic notions}
\subsection{Parabolic vector bundles}
\begin{defn}
\label{def: parabolic bun}
    A \textbf{parabolic vector bundle} over $\PP^1$ with parabolic structure at $S$ is a vector bundle $\cal V$ such that at each $p \in S$, $\cal V_p$ comes equipped with a complete flag, which we will denote as 
    $$
    F_1(p) \subset F_2 (p) \subset \ldots F_n(p) = \cal V_p.
    $$
    To each subspace $F_i (p)$ we associate a \textbf{parabolic weight} $\alpha_{n-i+1}(p)$, which must satisfy
    $$
    \alpha_1 (p) \leq \alpha_2 (p) \leq \ldots \leq \alpha_n(p) < \alpha_1 (p) + 1.
    $$
    Given a bundle with parabolic structure, we may also denote by $\cal V_{\alpha (p)}$ to be the largest subspace in the flag of $\cal V_p$ associated with weight $\alpha(p)$. 

    The \textbf{parabolic degree} of $\cal V$ is defined as
    $$
    \pardeg \calV = \deg \calV + \sum_{p \in S} \sum_{i=1}^n \alpha_i(p).
    $$

    The \textbf{parabolic slope} of $\cal V$ is defined as
    $$
    \parmu \cal V = \frac{\pardeg \cal V}{\rank \cal V}.
    $$
\end{defn}

Given a subbundle $\calW \subset \calV$ of rank $r$, we can give $\calW$ induced flags by intersecting $\cal W_p$ with $F_i (p)$. This subspace, call it $U \subset \cal W_p$ is then given the largest weight possible, that is, if $U = \calW_p \cap F_i(p) = \cal W_p \cap F_j(p)$, then we give $U$ the weight $\max \set{\alpha_{n-i+1} (p), \alpha_{n-j+1}(p)}$. This procedure may not necessarily produce a complete flag. For example, we could have $\dim \cal W_p \cap F_i (p) = m + \dim \cal W_p \cap F_{i+1} (p)$. In this case, we can complete the flag by adding in generic subspaces which must get the largest weight possible. In our example, we add in, say, $U_1 \subset \ldots \subset U_{m-1}$ between $\cal W_p \cap F_i (p)$ and $\cal W_p \cap F_{i+1} (p)$ where each $U_j$ gets weight $\alpha_{n-i} (p)$. 

Denote the the induced weights of $\cal W$ by $\beta_i (p)$, $i = 1, \ldots, r$. Then 
$$
\pardeg \calW = \deg \calW + \sum_{p \in S} \sum_{i=1}^r \beta_i (p).
$$

\begin{defn}
    A parabolic bundle $\calV$ is called \textbf{(semi)stable} if for all subbundles $\calW \subset \calV$, we have that 
    $$
    \parmu \cal W = \frac{\pardeg \calW}{\rank \calW} < \frac{\pardeg \calV}{\rank \calV} = \parmu \cal V \quad (\leq).
    $$
\end{defn}

\subsection{Parabolic Higgs bundles and Simpson's correspondence}
\label{section: par higgs}
\begin{defn}
    A \textbf{parabolic Higgs bundle} over $\PP^1$ consists of a pair $(\cal E, \theta)$, where $\cal E$ is a parabolic bundle over $\PP^1$ with parabolic structure at $S$ as above, and a morphism
    $$
    \theta : \cal E \to \cal E \otimes \Omega^1_{\PP^1} (\log S).
    $$
    Locally, over a point $p \in S$, given a local parameter $t$, we have
    $$
    \theta_p = A \frac{dt}{t} + A_0 + A_1 t + \ldots.
    $$
    We will call the $A$ in the above expansion the residue, written $\res_p \theta := A$. These residues will also be part of the data of a parabolic Higgs bundle. The residue is required to preserve the parabolic flags of $\cal E_p$, that is $\res_p \theta (\cal E_{\alpha_i (p)}) \subset \cal E_{\alpha_i (p)}.$

    A parabolic Higgs bundle is called \textbf{(semi)stable} if for all subbundles $\cal W \subset \cal E$, which are preserved by $\theta$, we have that
    $$
    \frac{\pardeg \calW}{\rank \calW} < \frac{\pardeg \cal E}{\rank \cal E} \quad (\leq).
    $$
\end{defn}

To state Simpson's correspondence, we need the notion of a filtered local system. Given a local system $\cal L$ on $U$ with base point $b$, we give $\cal L_b$ a filtration for each $p \in S$
$$
\cal L_b = \cal L_1^p \supset \cal L_2^p \supset \ldots \supset \cal L_n^p
$$
such that each $\cal L_i^p$ is associated with a number $\beta_i (p)$ with $\beta_1 (p) \leq \beta_2 (p) \leq \ldots \leq \beta_n (p)$. The local monodromy matrix $A_i$ (the image of $\gamma_i$ under the representation map) acts on the associated graded of $\cal L_b$ with the filtration labeled by $\beta_\bullet (p_i)$; call that matrix $\res (A_i)$. Now $\res A_i$ acts on the associated graded part $\cal L_j^{p_i} / \cal L_{j+1}^{p_+i}$ of $\cal L_b$. Putting $\res A_i$ into Jordan canonical form gives us a partition $P_i^{\beta_j(p_i), \lambda}$ which denotes the sizes of the Jordan blocks associated to eigenvalue $\lambda$.

On the Higgs bundle side, since $\res_p \theta$ preserves the filtration at each $p \in S$, this also acts on the associated graded. Call that induced map $Gr \res_p \theta$. We can take $\bar A(p_i) = Gr \res_{p_i} \theta$ at each $p_i \in S$, which acts on $\cal E_{i,p}$, and from this we can similarly get partitions $P_i^{\alpha_j(p_i), \tau}$ which denotes the sizes of the Jordan blocks of $\bar A (p_i)$ associated to its eigenvalue $\tau$. Under Simpson's correspondence, the partitions we get from $\res A_i$ are the same as the partitions we get from $\bar A(p_i)$ after doing the following change of labels:
\begin{align*}
    (\alpha, \tau = a + b \sqrt{-1}) & \mapsto (\beta = -2a, \lambda = e^{2 \pi \sqrt{-1} \alpha - 4 \pi b} ), \\
    (\beta, \lambda) &\mapsto (\alpha = \op{arg} (\lambda), \tau = -\beta/2 - \sqrt{-1} \log |\lambda| / 4 \pi).
\end{align*}

\subsection{Variation of Hodge Structures}
Over any curve $C$, the moduli space of Higgs bundles carries with it a natural $\C^\times$ action by $t \cdot (\cal E, \theta) = (\cal E, t \theta)$. If $C$ is compact, then the limits $t \to 0$ exist and give fixed points of the  $\C^\times$ action. Fixed points of this $\C^\times$ action gives rise to local systems coming from complex variation of Hodge structures. Over a noncomplete curve, we lose the ability to take limits of the $\C^\times$ action. However $\C^\times$ fixed points in this moduli of Higgs over the noncompact curve are still in correspondence with local systems coming from complex VHS as shown in \cite{simpson1990harmonic}. Such local systems are interesting, for instance, because they are local systems coming from geometry and all rigid local systems over $\PP^1 \sm S$ in the sense of \cite{katz1996rigid} are of this class. 

Parabolic Higgs bundles that come from such complex VHS break up in a special way. They correspond to systems of Hodge bundles defined below.
\begin{defn}
\label{def: shb}
    A \textbf{parabolic system of Hodge bundles} consists of a parabolic Higgs bundle $(\cal E, \theta)$ with parabolic structure at $S \subset \PP^1$ such that $\cal E$ decomposes as
    $$
    \cal E = \cal E_1 \oplus \cal E_2 \oplus \ldots \oplus \cal E_N,
    $$
    with $\theta|_{\cal E_i} : \cal E_i \to \cal E_{i+1} \otimes \Omega^1 (\log S)$. We will further say that a system of Hodge bundles has type $(r_1, \ldots, r_N)$ if $\rank \cal E_i = r_i$ for each $i$. 
    \label{hodge vec}

    Such parabolic systems of Hodge bundles are called \textbf{(semi)stable} if they are (semi)stable as parabolic Higgs bundles. 
\end{defn}

To check stability of such systems of Hodge bundles one only need to check subbundles compatible for the Hodge decomposition \cite{simpson1988constructing}. That is, we only need to check stability for bundles $\cal W = \cal W_1 \oplus \ldots \oplus \cal W_N$ with each $\cal W_i \subset \cal E_i$ and $\theta \mid_{\cal W} (\cal W_i) \subset \cal W_{i+1} \otimes \Omega^1 (\log S)$. We will call such subbundles \textbf{subsystems of Hodge bundles}. Furthermore, we can suppose that the flags of $\cal E_p$ respect the filtration as well, that is the subspaces of the flags of $\cal E_p$ breaks up into a direct sum of subspaces of $\cal E_{i,p}$. Thus, we can give each $\cal E_i$ the data of its own flags and weights. Now, since $\res_p \theta$ preserves the filtration, this means that $\res_p \theta ( \cal E_{1, \alpha} ) \subset \cal E_{2, \beta}$, where $\beta \geq \alpha$.

In this situation of local systems coming from $\C$-VHS, since we have that $\res_p \theta (\cal E_{i,p}) \subset \cal E_{i+1,p}$, after composing $\res_p \theta$ with itself $N$ times we must have that $(\res_p \theta)^N = 0$. Hence all the eigenvalues of $\res_p \theta$ are zero. Looking at the change of labels are described by the Simpson correspondence above, $(\alpha, \tau) \mapsto (\beta, \lambda)$, we must have that every $\beta = 0$ because, all eigenvalues the of $\res_p \theta$, $\tau = 0$. This means that there is no filtration on the level of local systems. Furthermore, because $\tau = 0$, we must also have that $\lambda$ must also lie on the unit circle. Thus, when we restrict ourselves to looking at local systems coming from $\C$-VHS, the correspondence is between stable parabolic systems of Hodge bundles of parabolic degree zero and local systems with local monodromies with eigenvalues of unit length.

\subsection{Strongly parabolic Higgs bundles}
In \cite{biswas2013parabolic}, a notion of \emph{strongly} parabolic Higgs bundles was defined. Adapted to our case and notation, it is defined in the following way.
\begin{defn}
    A \textbf{strongly parabolic Higgs bundle} will mean a parabolic Higgs bundle $(\cal E, \theta)$ as we defined such that the Higgs field $\theta$ will be such that
    $$
    \res_p \theta (\cal E_{\alpha(p)}) \subset \cal E_{\hat \alpha(p)}
    $$
    where $\hat \alpha(p)$ will denote the next highest weight after $\alpha(p)$, i.e. $\alpha(p) < \hat \alpha(p)$ with no other weights in between. In terms of flags, this means that $\res_p \theta$ takes each subspace to the next smallest subspace. In this case, we will also call $\theta$ as a \textbf{strongly parabolic Higgs field}.
\end{defn}
\begin{rmk}
    We note that in \cite{biswas2013parabolic} they call this notion as just a parabolic Higgs bundle. However, we will reserve this notion for Higgs bundles with the structure of parabolic vector bundles.
\end{rmk}

In terms of the local monodromies on the local system side, we note that asking for the Higgs field to be strongly parabolic will mean that once we take associated graded, $Gr \res_p \theta = 0$. So then going back to Simpson's correspondence to local systems, we will have that each local monodromy must be semi-simple, since the only partitions we will have are blocks of size one each. We will restrict ourselves to only looking at local systems with semisimple monodromies from now on, and hence we only look at strongly parabolic systems of Hodge bundles.

\subsection{Stability and the Harder-Narasimhan filtration}
\label{HN filt}
Any rank $n$ parabolic Higgs bundle $(\cal E, \theta)$ comes equipped with a unique Harder-Narasimhan filtration by subbundles
$$
\cal E = \cal E_1 \supsetneq \cal E_2 \supsetneq \ldots \supsetneq \cal E_h
$$
such that each $\cal E_i$ is preserved by $\theta$ (i.e. is a Higgs subbundle) with $\parmu(\cal E_1) < \parmu(\cal E_2) < \ldots < \parmu(\cal E_h)$ (see e.g. Section 3 of \cite{simpson1994moduli}). Consequently, if $(\cal E, \theta)$ is not a semistable parabolic Higgs bundle, then $\cal E_h$ will be the unique subbundle with maximal slope that violates the semistability inequalities. This $\cal E_h$ is unique, therefore when we do enumerative Gromov-Witten calculations, it suffices to check for only the cases where the Gromov-Witten number is one.

\section{The general algebro-geometric problem}
\label{section:general ag prob}
We now explain more about the general algebro-geometric/enumerative problem associated with the existence of stable parabolic systems of Hodge bundles. The big main question we context with is the following.

\begin{question}
Given the data of $(r_1, \ldots, r_N)$, $(d_1, \ldots, d_N)$, $(e_1, \ldots, e_{N-1})$ and weights $\set{\alpha^i_j (p) \mid 1 \leq j \leq r_i, i \in [N], p \in S}$, does there exist a semistable parabolic system of Hodge bundles $(\cal E, \theta)$ with
$$
\cal E = \cal E_1 \oplus \cal E_2 \oplus \ldots \oplus \cal E_N
$$
with $\rank \cal E_i = r_i$, $\deg \cal E_i = d_i$, $\rank \theta |_{\cal E_i} = e_i$, and such that $\cal E_i$ has flags at $p \in S$ with weights $\alpha^i_j (p)$ for $1 \leq j \leq r_i$?
\end{question}

This question can be tackled using enumerative geometry. Asking for such a stable bundle is the same as asking for all $\theta$-stable subbundles compatible with the Hodge decomposition to have parabolic slopes less than that of $(\cal E, \theta)$. In other words, suppose we have a subbundle $\cal W = \cal W_1 \oplus \cal W_2 \oplus \ldots \oplus \cal W_N$ where each $\cal W_i \subset \cal E_i$ is a subbundle of rank, say, $s_i$, and $\theta \mid_{\cal W} (\cal W_i) \subset \cal W_{i+1} \otimes \Omega^1 (\log S)$. Give each $\cal W_i$ the induced filtration from $\cal E_i$ as described in a similar process to usual parabolic bundles, and call these induced weights $\beta^i_j(p)$ for $1 \leq j \leq s_i$ and $i = 1, \ldots, N$ and $p \in S$. To ask for $(\cal E, \theta)$ to be semistable is to ask for all such subbundles $\cal W \subset \cal E$ the inequality
$$
\frac{\sum_{i=1}^N \paren*{\deg \cal W_i + \sum_{p \in S} \sum_{j=1}^{s_i} \beta^i_j (p)}}{\sum_{i=1}^N s_i} < \frac{\sum_{i=1}^N \paren*{\deg \cal E_i + \sum_{p \in S} \sum_{j=1}^{r_i} \alpha^i_j (p)}}{\sum_{i=1}^N r_i}
$$
holds.

In order to write down the inequalities above, we need to know when we have subsystems of Hodge bundles of type $(s_1, \ldots, s_N)$ that meets the flags of $\cal E$ in such a way as to give rise to parabolic weights $\set{\beta^i_j (p) \mid i \in [N], j \in [s_i], p \in S}$. This leads us to our enumerative problem.

\begin{question}
Given a parabolic system of Hodge bundles $(\cal E, \theta)$ with parabolic weights $\alpha^i_j (p)$ and flags as denoted above, when is there a $\theta$-stable Hodge subbundles $\cal W = \cal W_1 \oplus \ldots \oplus \cal W_N$ that meets the flags of $\cal E$ in such a way as to give $\cal W$ weights $\beta^i_j (p)$ as above.       
\end{question}

To tackle such an enumerative question, one should ask for irreducible components of the moduli space of all such Higgs bundles. Because stability is an open condition, to do the enumerative problem is to ask for general points of each component to be stable. So then we are lead to asking what data do we need to give ourselves an irreducible component? 

\begin{itemize}
    \item One easy example of an irreducible component is to ask for only one component in the Hodge decomposition and fixing any degree, i.e. ask for $\theta =0$ and fixing the degree of $\cal E$. So here we are looking at the moduli space of all bundles with fixed rank and degree, which is irreducible. In this case, we are really left with solving the unitary problem because the data we have is exactly the data of a regular parabolic vector bundle. 

    \item We can next ask for bundles of type $(1,n)$, so here $\cal E \simeq \cal L \oplus \cal V$, $\theta: \cal L \to \cal V \otimes \Omega^1 (\log S)$. We ask for $\theta \neq 0$ and so is injective, then this leaves us with the case of a bundle lying as a subsheaf inside another bundle. If we fix the numerical data of the $\cal E_i$'s (i.e. their ranks and degrees), this would then correspond to a quot scheme, which over $\PP^1$ is irreducible. 

    \item Another irreducible component we can find is the case where $(\cal E,\theta)$ is such that $\theta_i : \cal E_i \to \cal E_{i+1} \otimes \Omega^1 (\log S)$ makes each $\cal E_i$ into a subbundle of the next. So here we have flags
    $$
    \cal V_1 \subset \cal V_2 \subset \ldots \subset \cal V_N,
    $$
    where each $\cal V_i$ is an appropriated twisted $\cal E_i$. To then get irreducibility, we need to fix the numerical data of the degree and rank of each $\cal E_i$.

    \item We may also consider the dual version of the situation above. That is, consider a system of Hodge bundles $\cal E_1 \oplus \ldots \oplus \cal E_N$ where each $\theta_i := \theta \mid_{\cal E_i}: \cal E_i \to \cal E_{i+1} \otimes \Omega^1 (\log S)$ is a surjection. So, for example when $N = 2$, we are left with $\cal E_1 \to \cal E_2 \otimes \Omega^1 (\log S)$ is a surjection, in which case this is exactly the situation parametrized by certain quot schemes after fixing ranks and degrees. In the general situation, we have chains of surjections 
    $$
    \cal V_1 \twoheadrightarrow \cal V_2 \twoheadrightarrow \ldots \twoheadrightarrow \cal V_N
    $$
    and the space of which is also irreducible again after fixing degrees and ranks of each $\cal V_i$. 
\end{itemize}
In general, we may need to fix not just the data of the type of the system of Hodge bundles and degrees of the bundles, but also the data of the ranks of each $\theta_i : \cal E_i \to \cal E_{i+1} \otimes \Omega^1 (\log S)$. An interesting question would be to determine exactly what numerical conditions we need to fix in order to fix an irreducible component of the moduli space.

\section{Parabolic shifting operations}
\label{section:shifting}
For a parabolic bundle, there is a notion of a parabolic shifting operation. This operation does a cyclic shift on the weights and flags of the parabolic bundle and will be one of our main technical tools. 

Let $\calV$ be a rank $n$ bundle over $\PP^1$ with parabolic structure at point $p \in S$, with weights $\alpha_1 \leq \alpha_2 \leq \ldots \leq \alpha_n < \alpha_1 + 1$. Let $F_n \supset F_{n-1} \supset \ldots \supset F_1$ be the flag of $\calV_p$ coming from the parabolic structure with each $F_{n-i+1}$ associated with weight $\alpha_i$. Fix a local parameter $t$ at $p$. Then the shifted parabolic bundle $\tilde \calV$ consists of sections $s \in \calV(p)$ such that $ts |_p \in F_1$. $\tilde \calV$ coincides with  $\calV$ outside of $p$, so nothing happens away from $p$, therefore it suffices to only look at what happens when applying the shifting operation at the single point $p$. At $p$, we have that $\calV_p \to \tilde \calV_p$ has kernel $F_1$. We can define $\tilde F_i$ to be the image of $F_{i+1}$ of this map for $i < n$ and equal to $\tilde \calV_p$ for $i = n$. The shifted weights of $\tilde \calV$ at $p$ are called $\tilde \alpha_i$, and similar with the flags $\tilde \alpha_i = \alpha_{i-1}$ for $i > 1$ and $\tilde \alpha_1 = \alpha_n - 1$. For our shifted bundle $\tilde{\cal V}$, at point $p$ we have the flag
$$
\tilde F_n \supset \tilde F_{n-1} \supset \ldots \supset \tilde F_1 
$$
associated to weights 
$$
\tilde \alpha_1 = \alpha_n - 1 \leq \tilde \alpha_2 = \alpha_1 \leq \ldots \tilde \alpha_n  = \alpha_{n-1}.
$$
Explicitly, if $e_1, \ldots, e_n$ is a basis of $\cal V_p$ compatible with the filtration, that is, $F_1 = \braket{e_1} \subset F_2 = \braket{e_1, e_2} \subset \ldots \subset F_n = \braket{e_1, \ldots, e_n}$, then a basis of $\tilde{\cal V}_p$ would be $e_2, \ldots, e_n, e_1/t$ and we have
$$
\tilde F_1 = \braket{e_2} \subset \tilde F_2 = \braket{e_2, e_3} \subset \ldots \tilde F_{n-1} = \braket{e_2, \ldots, e_n} \subset \tilde F_n = \braket{e_2, \ldots, e_n, e_1/t}.
$$

We can apply this operation to parabolic systems of Hodge bundles, and it turns out that, after applying this operation, we are still left with a parabolic system of Hodge bundles. The main thing to verify when applying the parabolic shifting operations in this setting is what happens with the Higgs field $\theta : \cal E \to \cal E \otimes \Omega^1 (\log S)$. Locally at a point $p$ in $S$ with a chosen local parameter $t$, 
$$
\theta_p = A \frac{dt}{t} + B + \ldots,
$$
and the matrices $A, B$ must be linear in $t$ as well. Given our local basis $e_1, \ldots, e_n$ of $\cal E_p$ compatible with the filtration of $\cal E_p$ as above, once we shift, we are left with basis $e_2, \ldots, e_n, e_1/t$. Then $\tilde \theta$ is defined by extending $\theta$ to $e_1/t$ by factoring out the $1/t$ linearly. We will be much more explicit about how the shifting changes the Higgs field in various situations in the proof of Lemma \ref{lem: shifting preserves SHB}.

One subtlety around defining the shift operation for systems of Hodge bundles comes when we allow weights to be the same. Suppose we have a system of Hodge bundles given by $\cal E = \cal E_1 \oplus \cal E_2$, where $\rank \cal E_1 = r$ and $\rank \cal E_2 = s$, and $\theta: \cal E_1 \to \cal E_2 \otimes \Omega^1_{\PP^1} (\log S)$. Suppose at $p$, $\cal E_1$ has weights $\alpha_1 \leq \alpha_2 \leq \ldots \leq \alpha_r$ and $\cal E_2$ has weights $\beta_1 \leq \beta_2 \leq \ldots \leq \beta_s$. When we shift the whole bundle $(\cal E,\theta)$, we need to shift only the highest weight and hence move the lowest subspace of the flag at $p$. Therefore, we need to resolve the ambiguity whenever $\alpha_r = \beta_s$.  In this case, we will always take the subspace associated with $\beta_s$ to be the subspace we shift along, i.e. the smallest subspace of $\cal E_2$. 

This shifting operation is also invertible. Going back to the original situation of rank $n$ $\cal V$ with weights $\alpha_1 \leq \alpha_2 \leq \ldots \leq \alpha_n$, then if we shift $n$ times, we get that $\tilde{\cal V} = \cal V (p)$ with weights
$$
\alpha_1 - 1 \leq \alpha_2 - 1 \leq \ldots \leq \alpha_n - 1.
$$
We can then recover the original bundle $\cal V$ by twisting by $\cal O(-p)$ and the original weights by adding each weight of $\cal V$ by 1. Alternatively, we can simply reverse all the operations we've done for the shifting operation backwards. 

\begin{lem}
\label{lem: shifting preserves SHB}
Let $(\cal E, \theta)$ be a parabolic system of Hodge bundles over $\PP^1$ with parabolic structure at $S$. Let $\tilde{\cal E}$ be the image of $\cal E$ under the shift operator, and $\tilde \theta$ be the induced map. Then $(\tilde{\cal E}, \tilde \theta)$ is also a parabolic system of Hodge bundles. 
\end{lem}

\begin{proof}
    It suffices to prove this for systems of Hodge bundles of type $(r,s)$. Suppose $\cal E = \cal E_1 \oplus \cal E_2$, $\theta : \cal E_1 \to \cal E_2 \otimes \Omega^1 (\log S)$, $\rank \cal E_1 = r,$ and $\rank \cal E_2 = s$. Fix $p \in S$ as needed. Denote the parabolic weights of $\cal E_1$ at $p$ as
    $$
    \alpha_1 \leq \alpha_2 \leq \ldots \leq \alpha_r
    $$
    and the weights of $\cal E_2$ at $p$ as
    $$
    \beta_1 \leq \beta_2 \leq \ldots \leq \beta_s.
    $$
    Denote the flag of $\cal E_{1,p}$ as $F_1 \subset F_2 \subset \ldots \subset F_r$ and the flag of $\cal E_{2,p}$ as $G_1 \subset \ldots \subset G_s$, where $F_i$ gets weight $\alpha_{r - i + 1}$ and $G_j$ gets weight $\beta_{s - j + 1}$.
    
    Let $e_1, \ldots, e_r$ be a (local) basis for $\cal E_{1,p}$ compatible with its filtration, i.e. 
    $$
    F_1 = \braket{e_1} \subset F_2 = \braket{e_1, e_2} \subset \ldots \subset F_r = \braket{e_1, \ldots, e_r}
    $$
    Similarly, we pick a compatible basis $f_1, \ldots, f_s$ of $\cal E_{2,p}$ such that
    $$
    G_1 = \braket{f_1} \subset G_2 = \braket{f_1, f_2} \subset \ldots G_s = \braket{f_1, \ldots, f_s}.
    $$

    There are two possible cases for the shift operation: either $\alpha_r$ or $\beta_s$ is the highest weight, that is either $\alpha_r > \beta_s$ or $\alpha_r \leq \beta_s$. 

    \begin{noindent}
    \textbf{Case} $\alpha_r > \beta_s$:    
    \end{noindent}
    
    We first tackle when $\alpha_r$ is the highest weight. After shifting $\tilde{\cal E} = \tilde{\cal E_1} \oplus \cal E_2$. $\cal E_2$ remains unchanged after the shift operator with all the flags and weights remaining the same. The weights of $\tilde{\cal E}_1$ become now 
    $$
    \alpha_r - 1 \leq \alpha_1 \leq \ldots \leq \alpha_{r-1},
    $$
    and the compatible basis for $\tilde{\cal E}_{1,p}$ becomes $e_2, \ldots, e_r, e_1/t$ and the subspaces in the flag becomes
    $$
    \tilde F_1 = \braket{e_2} \subset \tilde F_2 = \braket{e_2, e_3} \subset \ldots \tilde F_{r-1} = \braket{e_2, \ldots, e_r} \subset \tilde F_r = \braket{e_2, \ldots, e_r, e_1/t}.
    $$
    Here $\tilde F_r$ gets weight $\alpha_r - 1$ and $\tilde F_{i}$ gets weight $\alpha_{n-i}$ for $i < r$. We need to check that $\res_p \tilde \theta$ still preserve this filtration. Note that the only thing we need to check in this situation is what happens to $e_1/t$. 
    
    Write $\theta_p$ as
    $$
    A \frac{dt}{t} + B \, dt + \ldots, 
    $$
    where we can write $A = (a_{ij})$ and $B = (b_{ij})$. Note that in the original unshifted bundle, $\alpha_r$ was the largest weight. This means that $\res_p \theta (e_1) = 0$, so $a_{i1} = 0$. Further applying $B$ to $e_1/t$ will now leave a pole. Hence we must have that $\res_p \tilde \theta_p (e_1/t) = \sum_{i=1}^s b_{i1} e_1$. Now, since $\alpha_r - 1 < \beta_1$, this means that actually $e_1/t$ is allowed to go to anything in $\braket{f_1, \ldots, f_s} = \cal E_{2,p}$, so $\tilde \theta$ still sends $e_1/t$ to the correct part of the flag in $\cal E_{2,p}$. Thus, $\res_p \tilde \theta$ preserves the filtration in this case.

    \begin{noindent}
        \textbf{Case} $\alpha_r \leq \beta_s$:
    \end{noindent}
    
    Next we need to check this result for when $\beta_s$ is the largest weight. In that case $\tilde{\cal E} \simeq{\cal E_1} \oplus \tilde{\cal E_2}$, where now $\cal E_1$ remains unchanged. The new weights of $\cal E_2$ at $p$ are
    $$
    \beta_s - 1 \leq \beta_1 \leq \ldots \leq \beta_{s-1},
    $$
    the new basis is $f_2, \ldots, f_s, f_1/t$, and the flag becomes
    $$
    \tilde G_1 = \braket{f_2} \subset \tilde G_2 = \braket{f_2, f_3} \subset \ldots \tilde G_{s-1} = \braket{f_2, \ldots, f_s} \subset \tilde G_s = \braket{f_2, \ldots, f_s, f_1/t},
    $$
    where similar to above, $\tilde G_s$ gets weight $\beta_s - 1$ while $\tilde G_i$ gets weight $\beta_{s-i}$ for $i < s$. Now we need to make sure that $\res_p \tilde \theta$ still preserve the filtration. Suppose $\res_p \tilde \theta (\cal E_{1, \alpha}) \subset \cal E_{2, \beta}$ such that $\cal E_{1, \alpha} = F_i$ and $\cal E_{2, \beta} = G_j$, so we have that $\res_p \theta (F_i) \subset G_j$, i.e. we have that $\braket{e_1, \ldots, e_i} \to \braket{f_1, \ldots, f_j},$ and when we shift, we must have that $\res_p \tilde \theta (\braket{e_1, \ldots, e_i}) \subset \braket{f_2, \ldots, f_j}$ for $j < s$ and if $j = s,$ we must have $\res_p \tilde \theta (\braket{e_1, \ldots, e_i}) \subset \braket{f_2, \ldots, f_s, f_1/t}$. The only case we need to worry about is when $j < s$.
    
    If $j < s$, then we encounter a problem only if any image of $\res_p \theta |_{F_i}$ has a component in $f_1 = t (f_1/t)$. But note that we can decompose $\res_p \theta$ by
    $$
    \res_p \theta (v)= \sum_{i=1}^{s-1} A_i (v) f_i + A_s (v) f_s,
    $$
    where each $A_i \in \cal E_{1,p}^\vee$. So then under the shift operation, $\tilde \theta$ becomes
    $$
    \sum_{i=1}^{s-1} A_i  f_i + A_s  t \paren*{\frac{f_s}{t}} + \ldots,
    $$
    which means then that the $f_s$ part is no-longer part of the residue, and hence $\res_p \tilde \theta (v) = \sum_{i=1}^{s-1} A_i (v) f_i$, and this map preserves the filtration because the original $\res_p \theta$ preserved the filtration, and hence every basis vector goes to the correct part of the filtration that it needs to go to. 
\end{proof}

\begin{cor}
    The parabolic shift operations preserve stability of systems of Hodge bundles.
\end{cor}

\begin{proof}
Let $(\cal E, \theta)$ be a parabolic system of Hodge bundles of rank $n$ and let $\cal W \subset \cal E$ be a subsystem of Hodge bundles or rank $r$, that is, $\cal W$ is a $\theta$-stable subbundle of $\cal E$. This in particular implies that $(\cal W, \theta|_{\cal W})$ is a system of Hodge bundles in its own right and it has the parabolic structure inherited as a parabolic subbundle. Fix a point $p$ where $\cal E$ has parabolic structure and suppose that the weights of $\cal E$ at $p$ are given by
$$
\alpha_1 \leq \alpha_2 \leq \ldots \leq \alpha_n,
$$
while $\cal W$ has induced weights $w_1 \leq \ldots \leq w_r.$ Applying the parabolic shifting operation at $p$, we have that $(\cal E,\theta)$ becomes $(\tilde{\cal E}, \tilde \theta)$, which is a parabolic system of Hodge bundles where $\deg \tilde{\cal E} = \deg{\cal E} + 1$ with parabolic weights at $p$ given now by
$$
\alpha_n - 1 < \alpha_1 \leq \ldots \leq \alpha_{n-1},
$$
so we have then that $\pardeg \cal E = \pardeg \tilde{\cal E}$ because the $-1$ in $\alpha_n -1$ cancels out with the $+1$ in $\deg \tilde{\cal E} = \deg \cal E + 1$.

Now let $F_1 \subset F_2 \subset \ldots \subset F_n$ be the parabolic flag at $\cal E_p$. One of two things may happen with $(\cal W, \theta|_{\cal W})$ after shifting along $F_1$. First, denote the resulting bundle by $(\tilde{\cal W}, \tilde \theta|_{\cal W})$. If $F_1 \subset \cal W_p$, then this means that after shifting, we have a new bundle $\tilde{\cal W}$ with $\deg \tilde{\cal W} = \deg \cal W + 1$ and $(\tilde{\cal W}, \tilde \theta)$ has parabolic weights at $p$ given by
$$
w_r -1 < w_1 \leq \ldots \leq w_{r-1}.
$$
Again by the same argument for $(\cal E, \theta)$ we have that $\pardeg \cal W = \pardeg \tilde{\cal W}$. If $F_1 \not\subset \cal W_p$, then $(\tilde{\cal W}, \tilde \theta) = (\cal W , \theta)$, so we still have $\pardeg \tilde{\cal W} = \pardeg \cal W$.

Now with a similar argument with the above lemma, one can show that the inverse shifting operation also preserves the property of being a system of Hodge bundles. This means that we have also a one-to-one correspondence between $\theta$-stable subbundles of $\cal E$ and $\tilde \theta$-stable subbundles of $\tilde{\cal E}$ via the shifting operation. We just saw that this shifting operation preserves parabolic degrees, and hence stability of $(\cal E, \theta)$ is preserved by the shifting operation.
\end{proof}

We will analyze this shifting operation again, more explicitly, when we analyze the $(1,2)$ case.

\section{The setup of the general \texorpdfstring{$(1,n)$}{(1,n)} problem}
Here our system of Hodge bundles decompose as $\cal E = \cal L \oplus \cal V$ as above where $\cal L$ is a line bundle and $\rank \cal V = n$ with $\theta (\cal L) \subset \cal V \otimes \Omega^1 (\log S)$. The main difficulty in the $(1,n)$ case is to determine when rank $r$ subbundles $\cal S \subset \cal V$ passes through prescribed Schubert cells in $Gr(r, \cal V_p)$. Since (semi)stability is an open condition, we are allowed to deform to the most generic case possible. In the unitary case, we are able to deform so that our bundle is an evenly split bundle, however, in this $(1,n)$ case, we have $\theta : \cal L \to \cal V \otimes \Omega^1 (\log S)$, which means that $\cal V$ must contain a certain line subbundle (or subsheaf). This becomes an obstruction to deformation, and hence we cannot assume that $\cal V$ is a generic rank $n$ bundle. Furthermore, the condition that $\res_p \theta$ preserves the filtration means that in certain configuration of weights, we need the flags of $\cal V_p$ to contain $\cal L_p$ in some special way, which means that we may not assume our flags to be generic in general. These obstructions means that we cannot apply older techniques of Gromov-Witten theory from the unitary case to this new situation as is. 

Since $\theta : \cal L \to \cal V \otimes \Omega^1 (\log S)$ and we need to consider only those bundles preserved by $\theta$ for stability of $(\cal L \oplus \cal V, \theta)$, there are two types of bundles to consider. 
\begin{enumerate}[label=(\textbf{\Roman*})]
    \item \label{sub of V} Every subbundle $\cal S \subset \cal V$ gets mapped to $0$ by $\theta$ and are in particular $\theta$-stable. Hence, we need to check every subbundle of $\cal V$. 

    \item \label{containing L} The other type of bundle to consider are those bundles containing the $\cal L$ factor in $\cal E = \cal L \oplus \cal V$. Thus, we need to consider subbundles of $\cal V \otimes \Omega^1 (\log S)$ that contains the image of $\cal L$. 
\end{enumerate}
To enumerate bundles of type \ref{sub of V}, we will need to do enumerative geometry involving the nongeneric bundle $\cal V$ equipped with nongeneric (quasi)parabolic structure, i.e. nongeneric flags at points $p \in S$. To enumerate bundles of type \ref{containing L}, we can look at suitable quotient bundles. 

Bundles of type \ref{containing L} must be of form $\cal L \oplus \cal S \subset \cal E = \cal L \oplus \cal V$, where $\cal L$ is the first factor of $\cal E$ and $\cal S$ is a subbundle of $\cal V$. Note that because of our conditions on weights, we may have $\theta : \cal L \to \cal V \otimes \Omega^1 (\log S)$ may not make $\cal L$ into a subbundle. For example, if $\alpha(p)$  are the weights of $\cal L$ and $\beta_1(p) \leq \beta_2 (p) \leq \ldots \beta_n (p)$ are the weights of $V$, and we have that $\beta_n (p) < \alpha(p)$ for some $p \in S$, then we must have that $\theta$ has a zero at $p$, and hence would fail to make $\cal L$ into a subbundle. We can insist on $\theta$ generic enough so that the only zeroes are those where $\beta_n(p) < \alpha(p)$. In this case, suppose that $p_1, \ldots, p_l$ are the points where $\beta_n (p) < \alpha(p)$ and $p_{l+1}, \ldots, p_s$ are the points where $\beta_n (p) \geq \alpha(p)$. Suppose $\cal L \simeq \cal O(-d)$ and let $\hat{\cal L}$ denote the saturation of $\cal L$, then we have that $\hat{\cal L} \simeq \calO(-d + l)$. If $\cal S$ is a subbundle of $\cal V$ such that $\cal S \otimes \Omega^1 (\log S)$ contains $\cal L$, then because $\cal S$ is a subbundle, it must contain the saturation $\hat{\cal L}$ as well. 

Now $\hat{\cal L} \simeq \cal O (-d + l)$ is a subbundle of $\cal V \otimes \Omega^1 (\log S)$ and, by arguments in \cite{belkale2008quantum}, we can assume that $\cal V \otimes \Omega^1 (\log S) / \hat{\cal L}$ is an evenly split bundle. In this case, suppose $\cal V$ has degree $v$, then $\cal V \otimes \Omega^1 (\log S) / \hat{\cal L}$ is an evenly split bundle of degree $v + n(s - 2) + d - l$. The strategy now is to control $\cal S$ by enumerating subbundles of $\cal V \otimes \Omega^1 (\log S) / \hat{\cal L}$ and then taking preimages, since such preimages are exactly those subbundles that contain $\cal L$.

\begin{lem}
\label{controlling subbundles by quotient}
Let $\cal W$ be a vector bundle and $\cal L \subset \cal W$ is a subbundle. Suppose $\cal W$ has flags at point $p$ given by $F_1 \subset F_2 \subset \ldots \subset F_n$. Denote by $\bar{\cal W} = \cal W / \cal L$ and $\bar F_1 \subset \bar F_2 \subset \ldots \subset \bar F_n$ the induced flag. Suppose $\cal S \subset \cal W$ is a subbundle that contains $\cal L$, and denote $\bar{\cal S} = \cal S / \cal L$, the image in the quotient $\bar{\cal W}$. Then
$$
\dim \paren*{\cal S_p \cap F_i} = \dim \paren*{\bar{\cal S}_p \cap \bar F_i} + \dim (F_i \cap \cal L_p).
$$
\end{lem}

\begin{proof}
Consider the intersection $\bar{\cal S}_p \cap \bar F_i$. Writing this out, we get
$$
\bar{\cal S}_p \cap \bar F_i = \frac{\cal S_p}{\cal L_p} \cap \frac{F_i + \cal L_p}{\cal L_p} = \frac{\paren*{\cal S_p \cap F_i + \cal L_p}}{\cal L_p} = \frac{\cal S_p \cap F_i}{F_i \cap \cal L_p}.
$$
Hence we find that $\dim \paren*{\cal S_p \cap F_i} = \dim \paren*{\bar{\cal S}_p \cap \bar F_i} - \dim \paren*{F_i \cap \cal L_p}$.
\end{proof}

\section{The \texorpdfstring{$(1,n)$}{(1,n)} case with generic weights}
In our attack for the $(1,n)$ case, we need to make some extra assumptions on the weights of our bundle. Give $\cal L$ weights $\alpha(p)$ for all $p \in S$ and $\cal V$ weights $\beta_1 (p) < \beta_2 (p) < \ldots < \beta_n (p)$, and fix the degrees of $\cal L$ and $\cal V$. Furthermore, we also ask for the $\alpha$'s to be distinct from the $\beta$'s, that is, we ask for all weights distinct. Suppose that the induced flags on $\cal V$ are called $F_1 (p) \subset F_2 (p) \subset \ldots \subset F_n(p) = \cal V_p$, where each $F_i (p)$ is associated to weight $\beta_{n - i + 1} (p)$ (remember the weights and flags have reverse ordering).

Because we assumed that all the weights of $\cal E$ are distinct, the shifting operation works in the exact same way as in the unitary case. And hence we still have a correspondence between subbundles of $\cal E$ and subbundles of $\tilde{ \cal E}$. In this case, there is no trouble if we just shift our weights to be such that
$$
\alpha(p) < \beta_1 (p) < \ldots < \beta_n(p)
$$
for all $p \in S$. In this case, we have that $\theta(\cal L)_p \subset F_n(p)$ for all $p$, and hence we can give $\cal V$ generic flags. Furthermore, since we have that $\alpha(p) < \beta_1 (p) < \beta_n(p)$ for all $p$, we necessarily have that $\theta$ has poles at every $p \in S$, and hence we will have no subtleties in terms of $\theta$ not making $\cal L$ into a subbundle in this case. Thus, we assume further that $\theta$ makes $\cal L$ into a subbundle of $\cal V \otimes \Omega^1 (\log S)$.

Set $\cal W = \cal V \otimes \Omega^1 (\log S)$. For simplicity, suppose that $\deg \cal W = \deg \paren*{\cal V \otimes \Omega^1 (\log S)} = 0$, since the other cases will not be substantially different; the only difference would be, if $\deg \cal W = w$, we would compare $\deg \cal L$ with $\ceil{w/n}$ instead of 0. There are now two situations to consider:
\begin{enumerate}[label=(\textbf{\Alph*})]
    \item $\deg \cal L \leq 0$, then in this case case, we can assume $\cal L \simeq \cal O(-d)$ for $d \geq 0$ and, in this case, we can deform $\cal W$ along with $\cal L$ to be an evenly split bundle, so we can assume $\cal W \simeq \calO^{\oplus n}$. So in this case we have $\calO(-d) \subset \cal O^{\oplus n}$. \label{small sub}

    \item If $\deg \cal L > 0$, so say $\cal L \simeq \cal O(d)$ for $d > 0$. In this case the most generic $\cal W$ (along with $\cal L$) we can deform to is $\cal W \simeq \calO (d) \oplus \cal G_{-d, n-1}$ (see Lemma \ref{lem: case b deformation}), where $\cal G_{-d, n-1}$ denotes the evenly split degree $-d$ rank $n-1$ bundle, so we can assume $\cal W$ is of this form already. \label{large sub}
\end{enumerate}
From here on, we assume that $\cal E$ has been deformed to the most generic case possible as described in the two cases above.

The case when $\deg \cal L \leq 0$ has already been solved before. In this case, we are left with counting subbundles of an evenly split bundle $\cal W$ meeting generic flags with prescribed incidence conditions, which has been solved using quantum cohomology of Grassmannians. So the interesting case is exactly the case when we have a large subbundle. 

The main goal of this section is to prove Theorem \ref{thm:(1,n)}. Recall the statement of the main theorem. Recall that $\pardeg \cal L + \pardeg \cal V = \pardeg \cal E$, and so fixing $\pardeg \cal L$ and $\pardeg \cal V$ automatically fixes $\pardeg \cal E$. Recall that we also made the simplifying assumption that $\deg \cal W = 0$.
\begin{thm*}
    A semistable system of Hodge bundles of type $(1,n)$ with weights $\alpha (p)$ and $\beta_1 (p) < \beta_2 (p) < \ldots < \beta_n (p)$ as above with fixed $\deg \cal L$ and $\deg \cal V$ exist if and only if the following conditions in the different cases of bundles are satisfied. 

\begin{enumerate}
        \item[\ref{small sub}] If $\deg \cal L \leq 0$, then the following two types of numerical conditions must be met.
        
        \begin{itemize}
        \item[\ref{sub of V}] For every $I^{p_1}, \ldots, I^{p_s} \subset [n]$ of size $r$ such that the usual Gromov-Witten number 
        $$
        \braket{\sigma_{I^{p_1}}, \ldots, \sigma_{I^{p_s}}}_\delta \neq 0,
        $$
        the inequality
        $$
        \frac{-\delta + r(2 - s) + \sum_{p \in S} \sum_{i \in I^p} \beta_{n-i+1} (p)}{r} \leq \frac{\pardeg \cal E}{n+1}
        $$
        holds.

        \item[\ref{containing L}] For every $J^{p_1}, \ldots, J^{p_s} \subset [n-1]$ of size $r-1$ such that the generalized Gromov-Witten number $\braket{\sigma_{J^{p_1}}, \ldots,  \sigma_{J^{p_s}}}_{\delta, d} \neq 0$, the inequality
        $$
        \frac{2d - \delta + r(2 - s) + \sum_{p \in S} \paren*{\alpha(p) + \beta_1(p) + \sum_{j \in J^p} \beta_{n-j+1} (p)}}{r+1} \leq \frac{\pardeg \cal E}{n+1}
        $$
        holds.
        \end{itemize}

        \item[\ref{large sub}] If $\deg \cal L > 0$, then the following two types of numerical conditions must be met. 
        \begin{itemize}
        \item[\ref{sub of V}] For every $I^{p_1}, \ldots, I^{p_s} \subset [n]$ of size $r$ such that $\braket{\sigma_{I^{p_1}}, \ldots, \sigma_{I^{p_s}}}_\delta \neq 0$, the inequality
        $$
        \frac{- \delta + r(2-s) + \sum_{p \in S} \sum_{i \in I^p} \beta_{n-i+1} (p)}{r} \leq \frac{\pardeg \cal E}{n+1}
        $$
        is satisfied.
        
        \item[\ref{containing L}] For every $J^{p_1}, \ldots, J^{p_s} \subset [n-1]$ of size $r-1$ such that $\braket{\sigma_{J^{p_1}}, \ldots, \sigma_{J^{p_s}}}_{\delta, d - w} \neq 0$, then the inequality 
        $$
        \frac{2d - \delta + r(2-s) + \sum_{p \in S} \paren*{\alpha(p) + \beta_1(p) + \sum_{j \in J^p} \beta_{n-j+1}(p)}}{r+1} \leq \frac{\pardeg \cal E}{n+1}
        $$
        is satisfied.
    \end{itemize}
    \end{enumerate}
\end{thm*}
We will prove this theorem by proving Theorem \ref{thm:small sub existence} which covers Case \ref{small sub} and Theorem \ref{thm:large sub existence} which covers Case\ref{large sub}, from which the result follows as an immediate corollary.

\subsection{Case where \texorpdfstring{$\cal L \subset \cal W$}{L in W} deforms to a generic bundle}
We will now tackle the Case \ref{small sub} situation. Recall that to check stability for $(\cal E, \theta)$, $\cal E = \cal L \oplus \cal V$, we need to check for two types of bundles: bundles that contain $\cal L$ and bundles that are contained purely in $\cal V$. To control bundles that contain $\cal L$, note that such bundles take the form
$$
\cal L \oplus \cal S \subset \cal E = \cal L \oplus \cal V
$$
where $\cal S \subset \cal V$ is a subbundle where $\theta (\cal L) \subset \cal S \otimes \Omega^1 (\log S)$. To enumerate such bundles, we can use the strategy outlined above by looking at subbundles $\cal V \otimes \Omega^1 (\log S)$ quotiented out by $\cal L$, which is given by $\cal G_{-d, n-1}$.

Let $F_\bullet (p)$ denote the flags of $\cal W$ at the point $p \in S$. Since we asked for $\alpha(p) < \beta_1 (p)$ for all $p$, we always have $\theta (\cal L)_p \subset F_n (p)$. And so the induced flag on $\bar{\cal W} := \cal W/\theta(\cal L) \simeq \cal G_{-d, n-1}$ for each $p \in S$ is given by $\bar F_1(p) \subset \bar F_2(p) \subset \ldots \subset \bar F_{n-1} (p)$ with associated weights $\beta_n (p) > \beta_{n-1} (p) > \ldots > \beta_2 (p)$ respectively. For subset $J^p = \set{j^p_1 < \ldots < j^p_{r-1}} \subset [n-1]$, we know there will be a subbundle $\bar{\cal S}$ of $\cal G_{-d,n-1}$ of degree $-\delta$ and rank $r - 1$ meeting $\bar F_\bullet (p)$ such that $\dim \bar{\cal S}_p \cap \bar F_{j^p_a} (p) \geq a$ if the generalized Gromov-Witten number (see \cite{belkale2022rigid} Definition 4.3.1) $\braket{\sigma_{J^1}, \ldots, \sigma_{J^s}}_{\delta, -d} \neq 0.$ Because of our arguments in Section \ref{HN filt}, it suffices to just check the cases where these Gromov-Witten numbers are nonzero.

\begin{prop}
    With the assumptions above, if $J^{p_1}, \ldots, J^{p_s} \subset [n-1]$ are cardinality $r-1$ subsets such that $\braket{\sigma_{J^{p_1}}, \ldots, \sigma_{J^{p_s}}}_{\delta, d} \neq 0$, then the inequality
    $$
    \frac{-2d - \delta + r(2 - s) + \sum_{p \in S} \paren*{\alpha(p) + \beta_1(p) + \sum_{j \in J^p} \beta_{n-j+1} (p)}}{r+1} \leq \frac{\pardeg \cal E}{n+1}
    $$
    is a necessary inequality  for the semistability of the parabolic system of Hodge bundles $(\cal E, \theta)$.
\end{prop}

\begin{proof}
As above, we let $\bar{\cal S}$ be the subbundle of $\cal G_{-d, n-1} \simeq \bar{\cal W}$ guaranteed by the nonvanishing of the generalized Gromov-Witten number $\braket{\sigma_{J^{p_1}}, \ldots \sigma_{J^{p_s}}}_{\delta, d}$. This lifts to a subbundle of $\cal W = \cal V \otimes \Omega^1 (\log S)$ containing $\theta(\cal L)$ of rank $r$; call this subbundle $\cal S$. Then $\cal L \oplus \cal S(2-s) \subset \cal E$ will be a $\theta$-stable subbundle of rank $r+1$, and hence we need to check $\parmu \cal L \oplus \cal S(2-s)$ against $\parmu \cal E$. 

Firstly, note that $\deg \bar{\cal S} = - \delta$ and $\bar{\cal S} \simeq \cal S / \cal L \simeq \cal S / \cal O(-d)$, this gives us that $\deg \cal S = - \delta - d$, and hence $\deg \cal S(2-s) = -d - \delta + r (2-s)$. Hence we get that
$$
\deg (\cal L \oplus \cal S(2-s)) = -2d - \delta + r(2 - s).
$$
Note that we can always take $\theta (\cal L)$ to be contained in $F_n(p)$ for all $p \in S$, this means that we can give $\cal S(2-s)$ the weights $\beta_1(p)$ at all times, and the other weights must come exactly by how $\bar{\cal S}$ meets $\bar F_1(p) \subset \ldots \subset \bar F_{n-1}(p)$ at each $p \in S$, i.e. the weights of $\cal S(2-s)$ are exactly given by $\beta_1(p)$ and $\beta_{n-j+1} (p)$ for $j \in J^p$ for each $p \in S$.
\end{proof}

Next we need to enumerate the subbundles purely of $\cal V$, equivalently those purely of $\cal W$. In such a case, we can simply use regular Gromov-Witten numbers to count such subbundles. 

\begin{prop}
    Let $I^p = \set{I^p_1 < I^p_2 < \ldots < I^p_r} \subset [n]$ be size $r$ subsets for $p \in S$. If $\braket{\sigma_{I^{p_1}}, \ldots, \sigma_{I^{p_s}}}_\delta \neq 0$, then the inequality
    $$
    \frac{-\delta + r(2-s) + \sum_{p \in S} \sum_{i \in I^p} \beta_{n-i+1} (p)}{r} \leq \frac{\pardeg \cal E}{n+1}
    $$
    is a necessarily inequality for the semistability of $(\cal E, \theta)$. 
\end{prop}

\begin{proof}
    Since $\cal W = \cal V \otimes \Omega^1(\log S)$ is the trivial rank $n$ bundle, subbundles of this are exactly enumerated by Gromov-Witten invariants, i.e. $\braket{\sigma_{I^{p_1}}, \ldots, \sigma_{I^{p_s}}}_\delta$ counts the number of degree $-\delta$ subbundles of $\cal W$. If this is nonzero, then we have a subbundle $\cal S \subset \cal W$ such that $\cal S_p$ meets the flag $F_\bullet (p)$ of $\cal W_p$ at $j^p_a$ for $a = 1, \ldots, r$, which would give $\cal S$ weights $\beta_{n-j+1} (p)$ for $j = j^p_1, j^p_2, \ldots, j^p_r$. Now $\cal S$ is a subbundle of $\cal W$ if and only if $\cal S(2-s)$ is a subbundle of $\cal V$, so then writing down the semistability condition for $\cal S(2-s)$ gives us the inequality we desire. 
\end{proof}

Combining the two propositions above together, we get the following theorem.
\begin{thm}
\label{thm:small sub existence}
    A semistable $(1,n)$ parabolic system of Hodge bundles $(\cal E, \theta)$, with parabolic weights $\set{\alpha(p), \beta_i(p)}_{i \in [n],p \in S}$ given as above and of Case \ref{small sub} exists if and only if the following conditions hold:
    \begin{itemize}
        \item[\ref{sub of V}] For every $I^{p_1}, \ldots, I^{p_s} \subset [n]$ of size $r$ such that the usual Gromov-Witten number 
        $$
        \braket{\sigma_{I^{p_1}}, \ldots, \sigma_{I^{p_s}}}_\delta \neq 0,
        $$
        the inequality
        $$
        \frac{-\delta + r(2 - s) + \sum_{p \in S} \sum_{i \in I^p} \beta_{n-i+1} (p)}{r} \leq \frac{\pardeg \cal E}{n+1}
        $$
        holds.
        
        \item[\ref{containing L}] For every $J^{p_1}, \ldots, J^{p_s} \subset [n-1]$ of size $r-1$ such that the Generalized Gromov-Witten number $\braket{\sigma_{J^{p_1}}, \ldots,  \sigma_{J^{p_s}}}_{\delta, d} \neq 0$, the inequality
        $$
        \frac{-2d - \delta + r(2 - s) + \sum_{p \in S} \paren*{\alpha(p) + \beta_1(p) + \sum_{j \in J^p} \beta_{n-j+1} (p)}}{r+1} \leq \frac{\pardeg \cal E}{n+1}
        $$
        holds.
    \end{itemize}
\end{thm}

\subsection{Case when the degree of \texorpdfstring{$\cal L$}{L} is large}
In Case \ref{large sub}, when $\deg \cal L > 0$, this requires more care. Here, we can deform $\cal L \subset \cal W$ so that $\cal W \simeq \calO(d) \oplus \cal G_{-d, n-1}$ by the following argument. In the below argument we do not make any assumptions on $\deg \cal W = 0$, which again will not make a difference.

\begin{lem}
\label{lem: case b deformation}
    In Case \ref{large sub}, we can deform $\cal L \subset \cal W$ so that $\cal W \simeq \cal L \oplus \cal G$, where $\cal G$ is an evenly split rank $n-1$ bundle.
\end{lem}

\begin{proof}
    After sufficiently twisting $\cal L \subset \cal W$ by $\cal O(-m)$ for $m >\!\!>0$, we can say that $\cal L \subset \cal W \subset \cal O^{\oplus N}$ for some $N$. Thus, we can view $\cal L \subset \cal W \subset \cal O^{\oplus N}$ as some element of a flag quot scheme $\cal Q$. Kim \cite{kim1996gromov} showed that such quot schemes are irreducible, so if we can show the locus where $\cal L \subset \cal W \subset \cal O^{\oplus N}$ is such that $W \simeq L \oplus \cal G$ where $\cal L \subset \cal W$ sits inside $\cal W$ as the first factor is open in this quot scheme, then we are done because semistability is an open condition and in an irreducible space, the intersection of two nonempty open sets is still nonempty open.

    Consider the bundle $\cal U$ over $\PP^1 \times \cal Q$ defined such that if $q$ represents the point $\cal L \subset \cal W$, 
    $$
    \cal U |_{\PP^1 \times q} = \cal U |_q \simeq \cal W / \cal L.
    $$
    It is a standard fact then that the locus in $\cal Q$ where $\cal U|_q$ splits as an evenly split bundle is open in $\cal Q$, say $\cal U|_q \simeq \cal G$. So we can restrict to looking at $0 \to \cal L \to \cal W \to \cal G \to 0.$
    If we can show that $\Ext^1 (\cal G, \cal L) = H^1 (\PP^1, \cal G^\vee \otimes \cal L) = 0$, then every extension $0 \to \cal L \to \cal W \to \cal G \to 0$ is trivial, thus split, and hence we must have $\cal W \simeq \cal L \oplus \cal G$.

    Say that $\deg \cal L = d$ and $\deg \cal W = w$, $\rank \cal W = n$. We know that $\cal G$ must have degree $w-d$ and rank $n-1$, so then $\cal G \simeq \bigoplus_{i=1}^{n-1} \cal O(e_i)$ where each $|e_i - e_j| \leq 1$. This forces $e_i \leq \ceil*{\frac{w-d}{n-1}}$. $\Ext^1 (\cal G, \cal L) = 0$ if each 
    $$
    \Ext^1 (\cal O(e_i), \cal L) = \Ext^1 (\cal O (e_i), \cal O(d)) \simeq H^1 (\PP^1, \cal O(d-e_i)) = 0,
    $$ 
    so if we can show that $d-e_i \geq 0$, then we are done. We know that
    $$
    d - e_i \geq d - \ceil*{\frac{w-d}{n-1}},
    $$
    so it suffices to show $d > \ceil*{\frac{w-d}{n-1}}$. We will show that $w/n > (w-d)/(n-1)$, and hence
    $$
    \ceil*{\frac{w}{n}} \geq \ceil*{\frac{w-d}{n-1}}.
    $$
    Since $d > \ceil{w/n} \geq w/n$, we must have that $-d < -w/n$, and so 
    $$
    \frac{w(n-1)}{n} - w =\frac{wn - w}{n} - w > -d.
    $$
    Rearranging, we then get that
    $$
    \frac{w}{n} > \frac{w-d}{n-1}.
    $$
    But by case \ref{large sub} assumptions, we must have that $d > \ceil{w/n}$, and thus $d > \ceil{(w-d)/(n-1)}$.
\end{proof}

Now we want to understand subbundles $\cal S \subset \cal W$. We will proceed along techniques similar to Section 9 of \cite{belgibmuk2015vanishing}. Here, we are looking for subbundles 
$$
\cal S \subset \cal O(d) \oplus \cal G_{-d,n-1}.
$$
Here there are two cases to consider:
\begin{enumerate}[label=(\alph*)]
    \item When $\calO(d) \subset \cal S \subset \calO(d) \oplus \cal G_{-d,n-1}$ and

    \item $\calO(d)$ is not in $\cal S$.
\end{enumerate}

The strategy is as follows: If $\cal O(d) \subset \cal S$, then since the inclusion $\calO(d) \hookrightarrow \calO(d) \oplus \cal G_{-d,n-1}$ is canonical (up to scalars), $\cal S$ is then fully determined by $\cal S / \calO(d) \subset \calO(d) \oplus \cal G_{-d,n-1} / \calO(d) \simeq \cal G_{-d,n-1}$. If $\cal S$ does not contain a $\cal O(d)$, then we will look at $\Ext^1 (\cal S, \calO(d) \oplus \cal G_{-d,n-1} / \cal S)$ which is the obstruction space controlling the deformations of $\cal S \subset \cal W$; we will show that in this case we have no obstructions and thus we can view this case as a deformation of a subbundle inside a generic bundle.

\begin{lem}
\label{lem: Ext(S, W/S) = 0}
    If $\cal O(d)$ is not in $\cal S \subset \cal W \simeq \calO(d) \oplus \cal G_{-d,n-1}$, then $\Ext^1 (\cal S, \cal W/\cal S) = 0$.
\end{lem}

\begin{proof}
    Consider $0 \to \cal S \to \cal W \to \cal W/ \cal S \to 0$. Because we are over a curve, all $\Ext^2$'s vanish, and thus we have a surjection $\Ext^1 (\cal S , \cal W ) \to \Ext^1 (\cal S , \cal W / \cal S).$ Thus, it suffices to show that $\Ext^1 (\cal S, \cal W) = 0$. 

    Suppose $\cal S \simeq \bigoplus_{i=1}^r \cal O(a_i)$. We have that 
    $$
    \Ext^1 (\cal S, \cal W) = H^1(\PP^1, \cal S^\vee \otimes \cal W),
    $$
    so it suffices to show that 
    \begin{align*}
    h^1 (\PP^1, \cal O(-a_i) \otimes \cal W) &= h^1( \PP^1, \cal O(-a_i) \otimes \paren{\cal O(d) \oplus \cal G_{-d,n-1}}) \\
    &= h^1 (\PP^1, \cal O(d-a_i)) + h^1 (\PP^1, \cal O(-a_i) \otimes \cal G_{-d,n-1})
    =0.    
    \end{align*}
    Since we have that $\cal O(d) \not \subset \cal S$, this means in particular that each $a_i < d$, and hence $h^1 (\PP^1, \cal O(-a_i) \otimes \cal O(d)) = h^1( \PP^1, \cal O (d-a_i)) = 0$ because $d-a_i \geq 0$. Thus, we are left with looking at $h^1 (\PP^1, \cal O(-a_i) \otimes \cal G_{-d,n-1})$. 

    Suppose that all the factors of $\cal G_{-d,n-1}$ are either $\cal O(e)$ or $\cal O(e+1)$, then 
    $$
    \Ext^1 (\cal O(a_i), \cal G_{-d,n-1})
    $$
    is a direct sum of vector spaces of the form
    $$
    H^1 (\PP^1, \cal O(-a_i) \otimes \cal O(e)) = H^1 (\PP^1, \cal O(e-a_i))
    $$
    and 
    $$
    H^1 (\PP^1, \cal O(-a_i) \otimes \cal O(e+1)) = H^1 (\PP^1, \cal O(e+1-a_i)).
    $$
    Because we asked for $\cal O(d) \not \subset \cal S$ and $\cal S$ is a subbundle, it must be then that the map $\cal S \subset \cal W$ gives a nonzero map $\cal O(a_i) \to \cal G_{-d,n-1}$, which then implies that $a_i \leq e+1$. This implies that $e-a_i \geq -1$ and $e+1 - a_i \geq 0$, which then implies that
    $$
    H^1 (\PP^1, \cal O(e-a_i)) = H^1 (\PP^1, \cal O(e+1-a_i)) = 0,
    $$
    and hence $\Ext^1 (\cal O(a_i), \cal G_{-d,n-1}) = 0$. This gives us $\Ext^1 (\cal S, \cal W) = 0$.
\end{proof}

This result shows that shows that we have no obstructions to deforming both $\cal W$ and $\cal S$. This means that by standard deformation of quot schemes (see e.g. Chapter 7 of \cite{hartshorne2010deformation} or Chapter I of \cite{kollar2013rational}), there exists a scheme $B$, a $\PP^1$-family $T \to B$, and bundles $\scr W$ and $\scr S \subset \scr W$ over $T$ such that over a point $b \in B$ we have that $\scr S |_b = \cal S \subset \scr W |_b = \cal W.$ Generically, $\scr W$ will be an evenly split bundle, and our goal will then be to count subbundles along the generic fiber instead.

\begin{lem}[Case \ref{large sub} type \ref{sub of V} inequalities]
    For each $p \in S$, let $I^p = \set{i^p_1 < i^p_2 < \ldots < i^p_r} \subset [n]$. If $\braket{\sigma_{I^{p_1}}, \ldots, \sigma_{I^{p_s}}}_\delta \neq 0$, then the inequality
    $$
    \frac{- \delta + r(2-s) + \sum_{p \in S} \sum_{i \in I^p} \beta_{n-i+1} (p)}{r} \leq \frac{\pardeg \cal E}{n+1}
    $$
    is necessary for the semistability of $(\cal E,\theta)$.
\end{lem}

\begin{proof}
    If we have a subbundle $\cal S \subset \cal W$ where $\cal O(d) \not\subset \cal S$, then by Lemma \ref{lem: Ext(S, W/S) = 0} we have that $\Ext^1 ( \cal S , \cal W / \cal S ) = 0$, so $\cal S \subset \cal W$ deforms in a family $\scr S \subset \scr W$ over base $B$ as above. Let $\cal Q \to B$ denote the quot scheme of subsheaves of $\scr W$ and let $\cal Q_b$ be the fiber of $\cal Q$ over $b$. Let $U \subset \cal Q$ be the locus of subbundles inside $\scr W$, and let $U_b$ be the fiber over $b$. We have a smooth map $U_b \to \prod_{p \in S} Gr(r, \cal W_p)$ by sending the subbundle to its fiber over each $p \in S$. Now if $\cal S_p$ is in some $\Omega_{I^p} (F_\bullet (p))$, then we can pull back to get an intersection of cycles on $U$, and by Kleiman transversality, we can assume that these cycles intersect transversely. Transversality is an open condition, and so we can look at the corresponding intersection over a generic fiber, which can be counted using Gromov-Witten invariants.

    Let $\xi \in B$ be a general point of $B$ such that $\scr W |_\xi \simeq \cal O_{\PP^1}^{\oplus n}$. We want to know if there exists a subbundle $\cal S \subset \cal W$ such that $\cal S_p \in \Omega_{I^p} (F_\bullet (p))$ for general flags $F_\bullet (p_1), \ldots, F_\bullet (p_s)$. By the arguments of the previous paragraph, we know such $\cal S$ exists if there is some subbundle of the same rank and degree of $\scr W|_\xi$ that lives in Schubert varieties of the given classes $\sigma_{I^p}$ for $p \in S$. These are exactly given by the numbers $\braket{\sigma_{I^{p_1}}, \ldots, \sigma_{I^{p_s}}}_\delta$, which if nonzero means there is a subbundle of the generic bundle of rank $r$ and degree $- \delta$ that hits generic flags at the $i$-th component for $i \in I^p$ for each $p \in S$, which then tells us such a rank $r$ degree $- \delta$ $\cal S \subset \cal W$ exists. 
    
    Since $\cal W = \cal V \otimes \Omega^1 (\log S)$, we need to look at $\cal S(2-s) \subset \cal V$, which has degree
    $$
    \deg \cal S(2-s) = -\delta + r (2-s). 
    $$
    At each $p \in S$, $\cal S$ gets weights $\beta_{n-i+1} (p)$
    for each $i \in I^p$, hence we get
    $$
    \pardeg \cal S(2-s) = - \delta + r(2-s) + \sum_{p \in S} \sum_{i \in I^p} \beta_{n-i+1} (p),
    $$
    from which the semistability inequality follows.
\end{proof}
\begin{rmk}
In the above argument, if we did not assume that $w=\deg \cal W = 0$, then instead of the usual Gromov-Witten number $\braket{\sigma_{I^{p_1}}, \ldots, \sigma_{I^{p_s}}}_\delta$, we would need to consider instead the generalized Gromov-Witten number $\braket{\sigma_{I^{p_1}}, \ldots, \sigma_{I^{p_s}}}_{\delta, -w}$. 
\end{rmk}

Next, we need to handle the subbundles where $\cal O(d) \subset \cal S$. In this case, we would need to look at subbundles $\cal S / \cal O(d) \subset \cal W / \cal O(d) \simeq \cal G_{-d, n-1}$. Here, $\cal G_{-d,n-1}$ comes equipped with the induced flags $\bar F_1(p) \subset \bar F_2(p) \subset \ldots \bar F_{n-1} (p)$ because $\cal O(d)_p = \theta(\cal L)_p$ lives (strictly) in $F_n (p)$. In this case counting subbundles of $\cal W$ that contain $\cal O(d)$ is also the same as counting bundles of type \ref{containing L}.
\begin{lem}[Case \ref{large sub} type \ref{containing L} inequalities]
    For each $p \in S$, let $J^p = \set{j^p_1 < \ldots < j^p_{r - 1}} \subset [n-1]$. If the generalized Gromov-Witten number
    $$
    \braket{\sigma_{J^{p_1}}, \ldots,  \sigma_{J^{p_s}}}_{\delta, d} \neq 0,
    $$
    then the inequality
    $$
    \frac{2d - \delta + r(2-s) + \sum_{p \in S} \paren*{\alpha(p) + \beta_1(p) + \sum_{j \in J^p} \beta_{n-j+1}(p)}}{r+1} \leq \frac{\pardeg \cal E}{n+1}
    $$
    is necessary for the semistability of $(\cal E, \theta)$.
\end{lem}

\begin{proof}
    If the generalized Gromov-Witten invariant $\braket{ \sigma_{J^{p_1}}, \ldots, \sigma_{J^{p_s}}}_{\delta, d} \neq 0$, then there exists a subbundle 
    $$
    \bar{ \cal S }\subset \cal G_{-d,n-1}
    $$
    of degree $-\delta$ and rank $r - 1$ such that $\bar{\cal S}_p$ meets the induced flags $\bar F_\bullet (p)$ at each each $j \in J^p$ for all $p \in S$. Then looking at the quotient $\cal W \twoheadrightarrow \cal W / \cal O(d) \simeq \cal G_{-d, n-1}$, there is a unique (up-to-scalar) lift of $\bar{\cal S}$ to a subbundle $\cal S \hookrightarrow \cal W$ containing $\cal O(d)$ of rank $r$. $\cal S \simeq \cal O(d) \oplus \bar{\cal S}$, so it has degree $d - \delta$. This bundle meets $F_\bullet (p)$ at $F_j (p)$ for each $j \in J^p$ and at $F_n (p)$ for each $p \in S$, which means that $\cal S$ gets weights $\beta_{n-j+1}(p)$ and $\beta_1 (p)$.  
    
    Now, if $\cal S \subset \cal W$, then we shift to $\cal S(2-s)$ to get a subbundle of $\cal V$. Thus, for such $\cal S$ containing $\theta (\cal L)$, we get the inequality
    $$
    \frac{d - \delta + r(2-s) + \sum_{p \in S} \paren*{\beta_1(p) + \sum_{j \in J^p} \beta_{n-j+1} (p)}}{r} \leq \frac{\pardeg \cal E}{n+1}.
    $$
    Since this is an inequality for a subbundle of $\cal V$, it was already covered in the previous proposition. Now the subbundles of $\cal E$ of type \ref{containing L} are also exactly of the form $\cal L \oplus \cal S(2-s)$ for such $\cal S$, which gives us the inequality
    $$
    \frac{2d - \delta + r(2-s) + \sum_{p \in S} \paren*{\alpha(p) + \beta_1(p) + \sum_{j \in J^p} \beta_{n-j+1} (p)}}{r+1} \leq \frac{\pardeg \cal E}{n+1}.
    $$
\end{proof}
\begin{rmk}
As in the previous lemma, when $w=\deg \cal W \neq 0$, the generalized Gromov-Witten number we need to look at instead would be $\braket{\sigma_{J^{p_1}}, \ldots, \sigma_{J^{p_s}}}_{\delta, d-w}$.
\end{rmk}

Once again, combining the two lemmas above, we get the following criteria for semistability of $(\cal E, \theta)$.
\begin{thm}
\label{thm:large sub existence}
    A $(1,n)$ semistable parabolic system of Hodge bundles $(\cal E, \theta)$ with parabolic weights $\set{\alpha(p), \beta_i (p)}_{i \in [n], p \in S}$ as above and of Case \ref{large sub} exists if the following criteria are met.
    \begin{itemize}
        \item[\ref{sub of V}] For every subsets $I^{p_1}, \ldots, I^{p_s} \subset [n]$ of size $r$ such that $\braket{\sigma_{I^{p_1}}, \ldots, \sigma_{I^{p_s}}}_\delta \neq 0$, the inequality
        $$
        \frac{- \delta + r(2-s) + \sum_{p \in S} \sum_{i \in I^p} \beta_{n-i+1} (p)}{r} \leq \frac{\pardeg \cal E}{n+1}
        $$
        is satisfied.

        \item[\ref{containing L}] For every subsets $J^{p_1}, \ldots, J^{p_s} \subset [n-1]$ of size $r-1$ such that $\braket{\sigma_{J^{p_1}}, \ldots, \sigma_{J^{p_s}}}_{\delta, d} \neq 0$, the inequality
        $$
        \frac{2d - \delta + r(2-s) + \sum_{p \in S} \paren*{\alpha(p) + \beta_1(p) + \sum_{i \in J^p} \beta_{n-j+1}(p)}}{r+1} \leq \frac{\pardeg \cal E}{n+1}
        $$
        is satisfied.
    \end{itemize}
\end{thm}

Combining Theorems \ref{thm:small sub existence} and \ref{thm:large sub existence}, we get Theorem \ref{thm:(1,n)} immediately.

\begin{ex}
    We now do a basic example using $S = \set{0,1,\infty}$.
    
    Suppose we are given the data of weights
    \begin{align*}
        & \alpha(0) < \beta_1 (0) < \beta_2 (0) < \beta_3 (0) < \beta_4 (0), \\
        & \alpha(1) < \beta_1 (1) < \beta_2 (1) < \beta_3 (1) < \beta_4 (1), \\
        & \alpha(\infty) < \beta_1 (\infty) < \beta_2 (\infty) < \beta_3 (\infty) < \beta_4 (\infty),
    \end{align*}
    where $\cal E = \cal L \oplus \cal V$ gets the weights in the same notation as above. Consider $\deg \cal L = 1$ and $\deg \cal V = 0$. Because $\alpha(p) < \beta_1 (p) < \beta_\infty (p)$ for all $p$, we can assume that $\theta$ takes $\cal L$ to a subbundle of $\cal W = \cal V (1)$. In this case, we can assume that $\cal V$ is evenly split. We assume also that $\pardeg \cal E = 0$, in which case we need that the sum of weights
    $$
    \sum_{p \in \set{0,1, \infty}} \alpha(p) + \beta_1(p) + \beta_2 (p) + \beta_3 (p) = 0.
    $$
    
    In this scenario, we have that $\theta$ takes $\cal O(1)$ to $\cal O(1)^{\oplus 4}$. Here, taking $\cal W / \cal L$ we get $\calO(1)^{\oplus 3}$. Since both bundles are twists of the trivial bundle, subbundles of $\cal W / \cal L$ are then completely determined by quantum cohomology of $Gr(r, 3)$ and subbundles of $\cal V$ by $Gr(r,4)$. 

    For subbundles of $\cal W / \cal L$, we can consider subsets $J^0, J^1, J^\infty$ of $[4-1]$ of size $r-1$. Here we can simply look at 
    $$
    \braket{\sigma_{J^0}, \sigma_{J^1}, \sigma_{J^\infty}}_\delta,
    $$
    which would count rank $r-1$ and degree $1-\delta$ subbundles of $\calO(1)^{\oplus 3}$ in this instance. For each such nonzero $\braket{\sigma_{J^0}, \sigma_{J^1}, \sigma_{J^\infty}}_\delta$, we would get the semistability inequality 
    $$
    1 -(\delta + r) + \sum_{p \in \set{0,1, \infty}} \paren*{\alpha(p) + \beta_1 (p) + \sum_{i \in J^p} \beta_{4-i} (p)} \leq 0.
    $$

    For subbundles of $\cal V$, we would need to consider subsets $I^0, I^1, I^\infty$ of $[4]$ and from here look at $\braket{\sigma_{I^0}, \sigma_{I^1}, \sigma_{I^\infty}}_\delta \neq 0$. Again, this would count all degree $-\delta$ subbundles of $\cal V = \calO^{\oplus 4}$. This leads to the semistability inequality
    $$
    \delta + \sum_{p \in \set{0,1,\infty}} \sum_{i \in I^p} \beta_{4-i+1} (p) \leq 0. 
    $$
    Hence now we need to consider when curve counts in $Gr(r,3)$ and $Gr(r,4)$ are nonzero. 
    
    For subbundles of $\cal W / \cal L$ both $Gr(2,3)$ and $Gr(1,3)$ are the same as quantum cohomology of $\PP^2$, so for an easy illustration, we can simply look at the $r = 1$ case, i.e. we are looking for rank two subbundles of $\cal W$ that contains $\cal L$. In this case, the subsets are $\set{1}, \set{2}$ and $\set{3}$. To do quantum cohomology computations, we convert these subsets to Young diagram data, which for $\PP^2$ just means a single integer $0 \leq \lambda \leq 2$. The conversion is given as follows:
    \begin{align*}
        \set{1} & \longleftrightarrow 2, \\
        \set{2} & \longleftrightarrow 1, \\
        \set{3} & \longleftrightarrow 0.
    \end{align*}
    We abuse notation and also denote $\sigma_\lambda$ for the Schubert class associated with the Young diagram $\lambda$. Using this notation, we can then list out the quantum multiplication table for $Gr(1,3)$:
    \begin{center}
    \begin{tabular}{c|c|c|c}
         & $\sigma_0$ & $\sigma_1$ & $\sigma_2$ \\
        \hline
        $\sigma_0$ & $\sigma_0$ & $\sigma_1$ & $\sigma_2$ \\
        $\sigma_1$ & $\sigma_1$ & $\sigma_2$ & $\sigma_0 q$ \\
        $\sigma_2$ & $\sigma_2$ & $\sigma_0 q$ & $\sigma_1 q$
    \end{tabular}
    \end{center}
    Reading off some select values, we see that $\braket{\sigma_1, \sigma_1, \sigma_2}_1 \neq 0$ and $\braket{\sigma_2, \sigma_2, \sigma_1}_1 \neq 0.$ Converting to subsets, we get $\braket{\sigma_{\set{2}}, \sigma_{\set{2}}, \sigma_{\set{1}}}_1 \neq 0$ and $\braket{\sigma_{\set{1}}, \sigma_{\set{1}}, \sigma_{\set{2}}}_1 \neq 0$. These would correspond to 
    $$
    -1 + \alpha(0) + \beta_1 (0) + \beta_2 (0) + \alpha(1) + \beta_1(1) + \beta_2(1) + \alpha(\infty) + \beta_1 (\infty) + \beta_3 (\infty) \leq 0
    $$
    and
    $$
    -1 + \alpha(0) + \beta_1(0) + \beta_3 (0) + \alpha(1) + \beta_1 (1) + \beta_3 (1) + \alpha(\infty) + \beta_1 (\infty) + \beta_2 (\infty) \leq 0
    $$
    respectively.

    For an inequality involving subbundles of $\cal V$, we would need to look at $Gr(r , 4)$. In this case, the only Grassmannian with interesting quantum cohomology is when $r = 2$, i.e. $Gr(2 , 4)$. One nonzero Gromov-Witten invariant we can get in this case would be $\braket{\sigma_{(2,2)}, \sigma_{(2,2)}, \sigma_{(2,2)}}_2 = 1$, which after translating to the language of subsets becomes 
    $$
    \braket{ \sigma_{\set{1,2}}, \sigma_{\set{1,2}}, \sigma_{\set{1,2}} }_2 = 1.
    $$
    Translating this into a subbundle of $\cal V$, we get a degree $-2$ subbundle of $\cal V \simeq \calO^{\oplus 4}$ that passes through the first and second parts of the flags of $\cal V$ at each $0, 1, \infty$. This corresponds to the semistability inequality
    $$
    -2 + \sum_{p \in \set{0,1,\infty}} \beta_4(p) + \beta_3 (p) \leq 0.
    $$
\end{ex}

\section{The \texorpdfstring{$(1,2)$}{(1,2)} case using parabolic shifting}
\label{section: (1,2)}
We can make less assumptions about the weights of our bundle in the $(1,2)$ system of Hodge bundles case. Instead of asking for all the weights to be distinct, we will make no assumptions on the weight data, and instead ask just for the local monodromies of the local systems to be semisimple.  

In this setting, our system of Hodge bundles decomposes again as $\cal E = \cal L \oplus \cal V,$ with$\theta: \cal L \to \cal V \otimes \Omega^1 (\log S)$. Now, we have $\rank \cal L = 1$ and $\rank \cal V = 2$. We also fix $\deg \cal L$ and $\deg \cal V$ to be some chosen integers. Let $\alpha(p)$ for $p \in S$ be the single weight for the line bundle $\cal L$ and $\beta_1(p) \leq \beta_2(p)$ be the weights for the rank 2 bundle $\cal V$. Suppose there are $m$ points $p \in S$ such that
$$
\beta_1(p) \leq \alpha(p) < \beta_2 (p).
$$
This means then that we have $m$ points where the flag of $\cal V$ at that point comes from $\cal L$. That is, if $F_p \subset \cal V_p \otimes \Omega^1 (\log S)_p$ is the rank 1 subspace coming from the parabolic structure on $\cal V$, then we necessarily have that $F_p = \theta(\cal L_p) \simeq \cal L_p.$

To simplify the situation further, suppose that $p_1, \ldots, p_l$ are the points in $S$ such that $\alpha(p) < \beta_2 (p).$ That is, these are exactly the points where $\theta$ has a pole. So then $p_{l+1}, \ldots, p_s$ are the points where $\theta$ has no pole. This means that we can simplify our situation to be such that
$$
\theta : \cal L \to \cal V \otimes \Omega^1 (p_1 + \ldots + p_l) \subset \cal V \otimes \Omega^1 (\log S).
$$ 
Simplifying this numerically, we have $\theta : \cal L \to \cal V(l-2) = \cal V \otimes \cal O(l-2)$. Now, since stability is an open condition, and the degrees of $\cal L$ and $\cal V$ are fixed, we can suppose then that $\theta$ is generic enough to make $\cal L$ into a subbundle of $\cal V (l-2)$. For simplicity of notation, set $\cal W = \cal V(l-2).$

We will always deform to the most generic bundle $\cal E = \cal L \oplus \cal V$ possible. There are two separate cases of vector bundles we can deform to. Suppose $\deg \cal W = \deg \cal V(l-2) = w$, then we have either that $\deg \cal L > \ceil{w/2}$ or $\deg \cal L \leq \ceil{w/2}$. If $d = \deg \cal L > \ceil{w/2}$, then we necessarily have that the most generic version of $\cal W$ we can deform to is $\cal W \simeq \cal O(d) \oplus \calO(w-d)$ (see Lemma \ref{lem: case b deformation}). And in the case that $\deg \cal L \leq \ceil{w/2}$, then $\cal L \subset \cal W$ can be allowed to be deformed so that $\cal W$ is an evenly split vector bundle of degree $w$. These two cases are essentially numerically the same in this situation as we will see by our arguments in the rest of this section, so we restrict ourselves to the case where $\cal W$ is generic. Further, any differences differ only by numerical conditions, and does not fundamentally change the problem, so for ease of notation, we restrict ourselves to $\cal W \simeq \cal O \oplus \cal O$. In summary, we simplify to the case of analyzing the situation of $\cal L \subset \calO^{\oplus 2}$, with $\cal L \simeq \cal O(d)$, $d \leq 0$.

Once again, there are two types of subbundles to consider, each giving a type of inequality.
\begin{itemize}
    \item[\ref{sub of V}] Subbundles of $\cal W$.
    
    \item[\ref{containing L}] $\cal L \oplus \theta(\cal L) \simeq \cal L \oplus \cal L$.
\end{itemize}
The main challenge here is to study subbundles of $\cal W$, since we already know how $\cal L \oplus \theta (\cal L)$ behaves using the fact that the residues of $\theta$ must preserve the filtration. 

The strategy to handle the $(1,2)$ case is to do the parabolic shifting operation at every point $p \in S$. Denote by $\tilde{\cal V}$ and $\tilde{\cal W}$ the bundles after applying the parabolic shifting operation at every point of $S$. By doing this, we essentially get rid of the parabolic structure in the following sense: Asking for numerical conditions means asking which stability inequalities we need to write. Such inequalities are given by subbundles of $\cal S \subset \cal V$. So if we shift along every point, we know then that the inequality coming from $\cal S$ is valid if and only if $\tilde{\cal S}$, the result of $\cal S$ after applying the shift operation, is still a valid subbundle $\tilde{\cal V}$. We can control this by directly reading the numerical data of $\pardeg \cal S$, i.e. $\deg \cal S$ and when $\cal S_p$ hit the special line subspace of $\cal W_p$ for each $p \in S$, which are controlled by which weights $\cal S$ has at $p$.

There is one small subtlety with regards to the parabolic shifting here: if we have that $\beta_1 (p) \leq \alpha (p) < \beta_2(p)$ (i.e. $\res_p \theta$ takes $\cal L_p$ to the special subspace of $\cal W_p$), then we will actually apply the shift operation on $\cal E$ twice so that we shift both $\cal L$ and $\cal V$ in this instance; otherwise, we only shift once. This is mainly a bookkeeping device to preserve the number of poles of the Higgs field and will not affect anything else because the second shift will only affect $\cal L$, leaving $\tilde{\cal V}$ and hence $\tilde{\cal W}$ intact.

For this $(1,2)$ case, we have the following theorem. Recall that fixing $\deg \cal L$ and $\deg \cal V$ and the weight data is enough to fix $\pardeg \cal E$. 
\begin{thm}[Existence of semistable type $(1,2)$ system of Hodge bundles]
\label{thm:(1,2)} 
Assume $\cal V$ is generic in the locus where $\cal V \otimes \Omega^1 (p_1 + \ldots + p_l)$ contains $\cal L$ as a subbundle. Suppose also that $\deg \cal L$ and $\deg \cal V$ are some fixed numbers.

A stable system of Hodge bundles of type $(1,2)$ where $\cal L$ takes weights $\alpha(p)$ and $\cal V$ takes parabolic weights $\beta_1 (p) \leq \beta_2(p)$ and $\cal E$ has the degree data as specified above exists if and only if the following conditions hold:

\begin{enumerate}
    \item[\ref{sub of V}] Let $I = (i(p))_{p \in S}$ where each $i(p) \in \set{1,2}$. Set $N = \abs{\set{p \in S \mid i(p) = 2}}$. Then for all $I$ and $\delta$ such that $\cal O(\delta + N)$ is a valid subbundle of $\tilde{\cal W}$, we have
    $$
    \delta - (l-2) + \sum_{p \in S} \beta_{i(p)} (p) \leq \frac{\pardeg \cal E}{3}.
    $$

    \item[\ref{containing L}] The inequality
    $$
    \frac{2 \deg \cal L - (l-2) + \sum_{p \in S} \alpha(p) + \sum_{i=1}^m \beta_1 (p_i) + \sum_{i=m+1}^s \beta_2 (p_i)}{2} \leq \frac{\pardeg \cal E}{3}
    $$
    must hold.
\end{enumerate}
\end{thm}

\subsection{Shifting operations in the (1,2) case}
\label{section: shifting and subbundles}
We will examine what the shifting operation does more explicitly in this $(1,2)$ scenario.

Let $\cal W$ be a rank 2 bundle over $\PP^1$. Over $p \in \PP^1$, give $\cal W_p$ a flag $F_1 \subset F_2 = \cal W_p = \C^2$. To complete the parabolic structure of $\cal W$, suppose we give $\cal W$ weights $\beta_1 \leq \beta_2$ at $p$ and for simplicity, we ignore all other points and say $\cal W$ only has parabolic structure at $p$. Then in this case $F_1$ has weight $\beta_2$ and $F_2$ has weight $\beta_1$.

Let $t$ be a local parameter of $p$ and give $\cal W_p = \C^2$ a basis compatible with the filtration $\set{e_1, e_2}$ such that
$$
F_1 = \C e_1, \quad F_2 = \C e_1 \oplus \C e_2.
$$
Locally around $p$, we have that $\cal W$ is trivialized as $\calO e_1 \oplus \calO e_2$. When we apply the shift operation around $p$, we get a new bundle of degree 1 higher, $\tilde{\cal W}$ with flag $\tilde F_1 = \C e_2 \subset \tilde F_2 = \C e_2 \oplus \C e_1/t$ and local trivialization $\calO e_2 \oplus \calO e_1 /t$. After shifting, we have that the parabolic weights of $\tilde{\cal W}$ are $\tilde \beta_1 = \beta_2 - 1 \leq \tilde \beta_2 = \beta_1$, with $\tilde F_1$ taking weight $\tilde \beta_2 = \beta_1$ and $\tilde F_2$ has weight $\tilde \beta_1 = \beta_2 - 1$.

\subsubsection{Shifting and subbundles}
Suppose $\cal S$ is a subbundle of $\cal W$ of degree $d$. Then locally around $p$, $\cal S$ is given by a section $a e_1 + b e_2$. When we shift, let $\tilde{\cal S}$ denote the subbundle $\cal S$ as viewed as a \emph{subsheaf} of $\tilde{\cal W}$. Locally, $\tilde{\cal S}$ is given by the section 
$$
at \frac{e_1}{t} + b e_2.
$$
The goal now is to understand the correspondence between subbundles of $\tilde{\cal W}$ and subbundles of $\cal W$ and also how the parabolic structure is affected by this correspondence. 

Suppose $\cal S$ is a subbundle of $\cal W$ that hits $F_1$. That means $\cal S_p = \C e_1$, and so locally $\cal S$ can be given by the section $e_1 + at e_2.$ In this case, once we shift, we get that $\tilde{\cal S}$ is given by 
$$
t e_1 / t + at e_2
$$
which means that $\tilde{\cal S}_p$ now fails to be a subbundle of $\tilde{\cal W}$. We can then replace $\tilde{\cal S}$ by its saturation. Doing this, we get that $\tilde{\cal S}$ is given locally by
$$
e_1 /t + a e_2
$$
which is a subbundle that does not hit $\tilde F_1 = \C e_2$, and hence $\tilde{\cal S}$ gets the weight coming from $\tilde F_2$, i.e. $\tilde{\cal S}$ gets weight $\tilde \beta_1 = \beta_2 - 1$. Since $\cal S$ hits $F_1$, this means that $S$ gets weight $\beta_2$. Since locally $\tilde S$ is given by $e_1/t + at e_2$, it means that we now have a pole of $\tilde{\cal S}$ at $p$, and hence $\deg \tilde{\cal S} = d + 1$. So then looking at the parabolic degrees of $\cal S$ and $\tilde{\cal S}$, we get that
$$
\pardeg \cal S = d  + \beta_2 \quad \text{and} \quad \pardeg \tilde{\cal S} = d + 1 + \tilde \beta_1 = (d + 1) + (\beta_2 - 1) = d + \beta_2.
$$
So in fact, we have that $\pardeg{\cal S} = \pardeg \tilde{\cal S}$ in this case.

Suppose now $\cal S$ is a subbundle of $\cal W$ that does not hit $F_1$. So then locally $\cal S$ is given by 
$$
a e_1 + e_2,
$$
and $\pardeg \cal S = d + \beta_1.$ Then applying the parabolic shift operation, we get that $\tilde{\cal S}$ is given locally by section
$$
at(e_1 / t) + e_2.
$$
Note that we do not have a pole at $p$, so in fact $\tilde{\cal S}$ is still a subbundle of $\tilde{\cal W}$. Note also that $\tilde{\cal S}$ must hit $\tilde F_1 = \C e_2$ at $p$, so we have that $\pardeg \tilde{\cal S} = d + \tilde \beta_2 = d + \beta_1$. So then, we have that once again $\pardeg S = \pardeg \tilde S$. This means that shifting in this way does not affect the parabolic degree. 

Next, we want to look at subbundles of $\tilde{\cal W}$. Suppose $\tilde{\cal S}$ is a subbundle of $\tilde{\cal W}$ of degree $d + 1$. Suppose first that $\tilde{\cal S}$ hits $\tilde F_1 = \C e_2$ at $p$. Then we can rewrite the local section giving $\tilde{\cal S}$ as
$$
e_2 + at e_1/t.
$$
Now note that $\tilde{\cal S}$ now does not have a pole at $p$ and so in fact is also a subbundle of $\cal W$. As above, to distinguish between $\tilde S$ as a subbundle of $\tilde{\cal W}$ and as a subbundle of $\cal W$, we again denote by $\cal S$ the subbundle of $\cal W$ given by $\tilde{\cal S}$. As a subbundle of $\cal W$, $\cal S$ is given by local section $e_2 + a e_1$, which in particular means that $\cal S$ must not $F_1 = \C e_1$ at $p$. In terms of parabolic structure, we can say that
$$
\pardeg \tilde{\cal S} = d + 1 + \tilde \beta_2 = d + 1 + \beta_1
$$
and 
$$
\pardeg{\cal S} = d + 1 + \beta_1.
$$
So in particular, $\pardeg S = \pardeg \tilde S$ once again in this situation.                                              

Now suppose $\tilde{\cal S}$ is a subbundle of $\tilde{\cal W}$ that does not hit $\tilde F_1 = \C e_2$. This means that locally $\tilde S$ is given by
$$
a e_2 + e_1/t.
$$
Consider a subbundle $\cal S$ of $\cal W$ given locally by 
$$
a t e_2 + e_1.
$$
This is a subbundle that hits $F_1 = \C e_1$, so then from the above arguments, we see that $\tilde{\cal S}$ arise as a result of taking the saturation of $\cal S$ when viewed as a subsheaf of $\tilde{\cal W}$. Then, same argument as above tells us that $\deg \cal S = \deg \tilde{\cal S} - 1 = d$. Under these circumstances, we see that 
$$
\pardeg \tilde{\cal S} = d + 1 + \tilde \beta_1 = d + 1 + \beta_2 - 1 = d + \beta_2 \quad \text{and} \quad \pardeg{\cal S} = d + \beta_2.
$$
Note in this case that $\pardeg \tilde{\cal S} = \pardeg \cal S$ as well. 

So in summary, what we have is that writing the stability inequalities for $\tilde{\cal W}$ implies every stability inequality for $\cal W$, as the subbundles of $\tilde{\cal W}$ corresponds exactly to the subbundles of $\cal W$ and as we checked, they all have the same parabolic degrees as well.

\subsection{Parabolic shifts on the splitting type}
\label{section: shift and splitting type}
Now we look at $\cal L \subset \cal W$, where $\cal L \simeq \calO (-d)$ and $\cal W \simeq \calO \oplus \calO$. Suppose that we have $m$ points $p_1, \ldots, p_m$ such that $\theta(\cal L)_p$ is the special subspace $F_p \subset \cal W_p$. Numerically speaking, $p_1, \ldots, p_m$ are the points where $\beta_1 (p) \leq \alpha(p) < \beta_2 (p)$. After doing parabolic shifts along the $m$ points, we shift the bundle to
$$
\calO(m-d) \hookrightarrow \tilde{\cal W}.
$$
The question is then what happens to the splitting type of $\tilde{\cal W}$? As a first step, we look at when $\tilde{\cal W}$ is guaranteed to fail to be evenly split. 

\begin{lem}
\label{lem: shifting not generic ineq}
    If $m > 2d + 1$, then $\tilde{\cal W}$ fails to be evenly split. 
\end{lem}

\begin{proof}
    Shifting along the $m$ points, we have 
    $$
    \calO (m-d) \hookrightarrow \tilde{\cal W}
    $$
    where $\tilde{\cal W}$ has degree $m$. Suppose that $\tilde{\cal W} \simeq \calO (a) \oplus \calO(b)$ and that $a \geq b$. In order for $\tilde{\cal W}$ to be evenly split, we would need $a \leq \ceil{m/2}$. For convenience, we can simply recast this inequality as 
    $$
    a \leq m/2 + 1/2.
    $$
    For $\calO(m-d)$ to even inject into $\tilde{\cal W}$, we would need 
    $$
    m-d \leq a \leq m/2 + 1/2.
    $$
    Thus we have $m-d \leq m/2 + 1/2$ which implies that $m \leq 2d + 1$. So if $m > 2d + 1$ $\tilde{\cal W}$ cannot be generic. 
\end{proof}

\begin{ex}
    For a simple example of the situation, consider $\cal L = \calO (-1) \hookrightarrow \cal W = \calO \oplus \calO$. This corresponds to a degree 1 map $\PP^1 \to \PP^1$, which we know can send 3 points to any 3 points. That means for 3 points, we can guarantee that the subspaces coming from $\cal L$ are generic. We can recover this 3 from Lemma \ref{lem: shifting not generic ineq} because, here $d = 1$, so $2d + 1 = 3$.

    Suppose we have $m > 2d + 1$ points coming from $\cal L$. That is, we have $m$ points, say $p_1, \ldots, p_m$ in $S$ such that $\theta (\cal L)_p = F_p \subset \cal W_p$ for each $p = p_1, \ldots, p_m$. We want to apply the parabolic shifting operation along these $m$ points. We can actually see very explicitly what happens to the splitting type. 

    We do the shifts one at a time. We can suppose that $\theta(\cal L) \subset \cal W$ generically enough such that at any three points, $\theta(\cal L)_p$ sits as a generic subspace inside $\cal W_p$. Suppose we shift first along $p_1, p_2, p_3$.
    \begin{enumerate}[label=]
        \item No shifts, i.e. the original situation: $\calO(-1) \hookrightarrow \calO \oplus \calO$
        \item Shift at $p_1$: $\calO \hookrightarrow \calO(1) \oplus \calO$
        \item Shift at $p_2$: $\calO (1) \hookrightarrow \calO(1) \oplus \calO(1)$
        \item Shift at $p_3$: $\calO (2) \hookrightarrow \calO(2) \oplus \calO(1)$
    \end{enumerate}
    Now by the above lemma, we see that afterwards from here, we only get non-evenly split bundles. If we shift along a fourth point $p_4$, we would get
    $$
    \calO(3) \hookrightarrow \calO(3) \oplus \calO(1).
    $$
    Now this must be the case because in order for $\calO(3)$ must inject into $\tilde{\cal W}$, one of the summands of $\tilde{\cal W}$ must be $\cal O(a)$ for $a \geq 3$. Further, we have that $\cal W \subset \tilde{\cal W} \subset \cal W(p_4)$, which numerically translates to
    $$
    \calO(2) \oplus \calO(1) \subset \tilde{\cal W} \subset \calO(3) \oplus \calO(2).
    $$
    Note that $\deg \tilde{\cal W} = \deg \cal W + 1$. So then in order for $\calO(3)$ to inject into $\tilde{\cal W}$, the only possible choice for $\tilde{\cal W}$ is $\calO(3) \oplus \calO(1)$.

    It is then easy to see that continuing further, we get
    $$
    \calO(m-1) \hookrightarrow \calO(m-1) \oplus \calO(1)
    $$
    for the general case after shifting along $m$ points lying on the $\cal L = \calO(-1)$ sitting inside the $\cal W = \calO \oplus \calO$.
\end{ex}

We can suppose that $S \subset \PP^1$ and $\theta(\cal L) \subset \cal W$ are generic enough such that if we shift along less than $2d+1$ points, then the resulting shifted bundle $\tilde{\cal W}$ is still generic.

\begin{lem}
    For $m > 2d+1$, if we shift $\cal L = \calO(-d) \subset \cal W= \calO \oplus \calO$ along the $m$ points, the splitting type of the resulting bundles must be
    $$
    \tilde{\cal L} = \calO(m-d) \subset \tilde{\cal W} = \calO(m-d) \oplus \calO(d).
    $$
\end{lem}

\begin{proof}
    We can choose $2d + 1$ among the $m$ points generic enough so that after shifting along them we get
    $$
    \cal L' = \calO(2d + 1 - d) = \calO(d+1) \subset \cal W'
    $$
    where $\cal L'$ and $\cal W'$ will denote $\cal L$ and $\cal W$ after shifting along the $2d+1$ points. The degree of $\cal W'$ is $2d+1$ and it is generic, thus 
    $$
    \cal W' \simeq \calO(d+1) \oplus \calO(d).
    $$

    We can assume WLOG that $m = 2d+2$. So then shifting one more time, say along $p$, we then have
    $$
    \tilde{\cal L} = \calO(2d+2-d) = \calO(d+2)
    $$
    and $\tilde{\cal W}$ has degree $2d+2$ with $\cal W' \subset \tilde{\cal W} \subset \cal W(p)$. We have that
    $$
    \cal W \simeq \calO(d+1) \oplus \calO(d) \quad \text{and} \quad \cal W(p) \simeq \calO(d+2) \oplus \calO(d+1).
    $$
    Then in order for $\tilde{\cal L} = \calO(d+2)$ to inject into $\tilde{\cal W}$, the only option available is
    $$
    \tilde{\cal W} = \calO(d+2) \oplus \calO(d).
    $$
    Continuing this way for arbitrary $m > 2d + 1$, we would get
    $$
    \calO(m-d) \hookrightarrow \calO(m-d) \oplus \calO(d).
    $$
\end{proof}

Now $p_1, \ldots, p_m$ are the only points that impose special conditions on the flags of $\cal V$, and hence at all the other points, we may choose the flags to be as generic as we like. Then, shifting along all the other points will make our bundle less special. If we shift along $p_{m+1}, \ldots, p_l$, we would get then that splitting type of $\tilde{\cal W}$ would be $\calO(m-d) \oplus \calO (d + l - m)$ if $d + l - m \leq m - d$ and otherwise $\tilde{\cal W}$ would be a generic rank 2 degree $l$ bundle. However, we must now be careful about shifting at all the other points.

For $p = p_{l+1}, \ldots, p_s$, we have that $\beta_2(p) \leq \alpha(p)$, which means that $\theta$ does not admit poles at these points, but then after shifting, we have that $\tilde \alpha(p) < \tilde \beta_2 (p)$, which means then that we suddenly get a new pole. This means that extra book keeping is required here. Fortunately, what happens is that we get an extra pole of $\tilde \theta_p$, and so we will consider instead the shifted bundle but now twisted with $\cal O(p)$. For example, when we have $\tilde{\cal W} \simeq \cal O(m-d) \oplus \cal O(d+l-m)$, i.e. we have shifted at all the $l$ points where $\theta$ already has a pole, if we shift one more time and twist by $\cal O(p_{l+1})$, we now only consider $\tilde{\cal W}(p_{l+1}) \simeq \paren*{\cal O(m-d) \oplus \cal O(d+l-m) } \otimes \cal O(1).$ We will abuse notation to denote by $\tilde{\cal W}$ what happens to $\cal W$ after shifting at all points $p \in S$ as described in the above procedures, including this extra bookkeeping of adding in extra poles. 
 
\subsection{The semistability inequalities we need}

Now we want to write down the inequalities guaranteeing semistability of our bundle $\cal E = \cal L \oplus \cal V$. Recall that we defined $\cal W = \cal V(l - 2 )$. Denote the correction term by $Q = l - 2$ for simplicity of notation. Recall that to check stability of $\cal E$ requires us to check for two types of subbundles:
\begin{itemize}
    \item[\ref{sub of V}] Subbundles purely of $\cal W$ (which are the same as subbundles of $\cal V$ after an appropriate twist).
    
    \item[\ref{containing L}] The subbundle $\cal L \oplus \paren*{\theta (\cal L)(-Q)}$.
    
\end{itemize}
The first case gives us an easy inequality to write, since by our assumptions, $\theta(\cal L)$ hits the special subspace of $\cal V$ only when $\beta_1(p) \leq \alpha(p) < \beta_2(p)$, which by assumption only happens when $p = p_1, \ldots, p_m \in S$. This means that the inequality we need to write here is
$$
\frac{2\deg \cal L - Q + \sum_{p \in S} \alpha(p) + \sum_{i=1}^m \beta_2(p_i) + \sum_{i = m+1}^s \beta_1(p_i)}{2} \leq \frac{\pardeg \cal E}{3}.
$$

We now only need to consider the subbundles of the second type, i.e. subbundles of $\cal V$, which are the same as subbundles of $\cal W$. Now note that we can write down every single potential inequality to check for the second type of bundle already:
$$
\delta - Q + \omega \leq (\pardeg \cal E)/3,
$$
where here we use $\omega$ to denote the sum of weights that any potential degree $\delta$ subbundle of $\cal W$ may hit. The question that remains is when do we write such an inequality? Denote by $\tilde{\cal W}$ the result of parabolically shifting $\cal W$ along every point of $S$. By the above results, we know exactly what the splitting type of $\tilde{\cal W}$ is. 

\begin{prop}
\label{prop: shifted ineq}
Let $N$ denote the number of times $\beta_2(p)$ appears in $\omega$, then the inequality $\delta - Q + \omega \leq (\pardeg \cal E)/3$ is a necessary inequality to guarantee the stability of $\cal E$ if and only if $\calO(\delta + N)$ is a valid subbundle of $\tilde{\cal W}$.
\end{prop}

\begin{proof}
    We need to write such an inequality if and only if there exists a degree $\delta$ subbundle $\cal S$ of $\cal W$ such that $\cal S_p = F_1 (p)$ for every $\beta_2(p)$ that appears in $\omega$. Such a subbundle exists if and only if, after shifting, we have a subbundle $\tilde{\cal S} \simeq \calO(\delta + N)$ of $\tilde{\cal W}$.
\end{proof}

\begin{rmk}
    We can rephrase the above Proposition by looking at the subbundles of $\cal V$ and compared with the subbundles of $\tilde{\cal V}$ instead. That is, for any $\cal O(\delta)$ being a valid subbundle of $\cal V$, the inequality $\delta + \omega \leq \pardeg \cal E$ is a necessary inequality to guarantee the semistability of $\cal E$ if and only if $\cal O(\delta + N)$ is a valid subbundle of $\tilde{\cal V}$. 
\end{rmk}

\begin{rmk}
    $N$ in the above proposition is exactly the $N$ given in Theorem \ref{thm:(1,2)}. That is, we set
    $$
    \omega = \sum_{p \in S} \beta_{i(p)} (p)
    $$
    for $I = (i(p))_{p \in S}$ as in the Theorem. And so $N = \abs{\set{p \in S \mid i(p) = 2}}$
\end{rmk}

\subsection{Proof of Theorem \texorpdfstring{\ref{thm:(1,2)}}{(1,2)}}
\begin{proof}
    We are given the data of weights $\alpha(p), \beta_1(p) \leq \beta_2(p)$ for $p \in S$ and fixed degrees $D = \deg \cal E$ and $d = \deg \cal L$. Here $\cal L \simeq \cal O (d)$, and we have that $\cal V$ must have degree $v = D - d$. We deform $\cal V$ along with $\cal L$ to be as generic as possible, i.e. we ask for $\cal W = \cal V \otimes \cal O(l-2)$ to be generic containing $\cal L$ as a subbundle. Thus $\cal V$ is evenly split of degree $v$ if $d \leq \ceil{w/2}$ or $\cal W \simeq \cal O(d) \oplus \cal O(w-d)$ otherwise. 

    We then apply the shifting operation along every point of $S$ as described above. This gives us $\tilde{\cal V}, \tilde{\cal W} \simeq \tilde{\cal V} (l - 2)$ as described in Section \ref{section: shift and splitting type}. In such a situation, the inequalities we need are described in the above. 

    For inequalities of type \ref{sub of V}, we are looking at (line) subbundles purely of our rank 2 bundle $\cal W$, or equivalently, purely of $\cal V$. Proposition \ref{prop: shifted ineq} gives exactly the inequalities of this type. That is, for all $I = (i(p))_{p \in S}$ with $i(p) \in \set{1,2}$ and for all $\delta$ such that $\cal O(\delta + N)$ is a valid subbundle of $\tilde{\cal W}$, we must have the inequality
    $$
    \delta - (l-2) + \sum_{p \in S} \beta_{i(p)} (p) \leq \frac{\pardeg \cal E}{3}.
    $$

    For the inequality of type \ref{containing L}, we are looking at the subbundle $\cal L \oplus \theta(\cal L) (2-l)$. We get the inequality
    $$
    \frac{2 \deg \cal L - (l-2) + \sum_{p \in S} \alpha(p) + \sum_{i=1}^m \beta_2(p_i) + \sum_{i = m+1}^s \beta_1(p_i)}{2} \leq \frac{\pardeg \cal E}{3}.
    $$

\end{proof}

\subsection{Examples}
\begin{ex}
    Suppose we have $\cal L \simeq \cal O(1)$ and $\deg \cal V = 0$ and $S = \set{0, 1, \infty}$. Suppose that our weights are of the form
    \begin{align*}
        & \alpha(0) < \beta_1 (0) < \beta_2 (0), \\
        & \beta_1 (1) < \alpha(1) < \beta_2 (1), \\
        & \beta_1 (\infty) < \alpha (\infty) < \beta_2 (\infty).
    \end{align*}
    We suppose that $\pardeg \cal E = 0$, so we have that
    $$
    \sum_{p \in S} \alpha(p) + \beta_1 (p) + \beta_2 (p) = -1.
    $$
    Here, our correction term $Q = 3 - 2 = 1$, so we can suppose $\cal W \simeq \cal O(1) \oplus \cal O(1)$, and we have $\theta : \cal O(1) \to \cal O(1) \oplus \cal O(1)$. The semistability inequality we need to write then for the subbundle $\cal L \oplus \theta(\cal L) (-Q)$ is 
    $$
    1 + (\alpha(0) + \alpha(1) + \alpha(\infty)) + \beta_2(1) + \beta_2 (\infty) + \beta_1 (0) \leq 0.
    $$
    
    We then apply the shift operation at every point $p \in S$ to get that $\tilde{\cal V} = \cal O(1) \oplus \cal O(2)$. We can then look what happens under the shifting operation to subbundles of $\cal V$ and compare with $\tilde{\cal V}$. 

    We need that $\cal O(\delta)$ to be a valid subbundle of $\cal V = \cal O \oplus \cal O$, so here $\delta = 0$ will do (as it has maximal degree). Then we need that $\cal O(\delta + N)$ be a valid subbundle of $\tilde{\cal V} \simeq \calO(1) \oplus \cal O (2)$, and here $N = 2$ is one. The semistability inequalities that we get using $\delta = 0$ and $N = 2$ are 
    \begin{align*}
        & \beta_2 (0) + \beta_2 (1) + \beta_1 (\infty) \leq 0, \\
        & \beta_2 (1) + \beta_2 (\infty) + \beta_1 (0) \leq 0, \\
        & \beta_2 (0) + \beta_2 (\infty) + \beta_1 (1) \leq 0.
    \end{align*}
\end{ex}

\begin{ex}
    We take $\cal L$, $\cal V$ and $S$ the same as the previous example. This time we ask for weights such that
    \begin{align*}
        & \beta_1 (0) < \alpha(0) < \beta_2 (0), \\
        & \beta_1 (1) < \alpha(1) < \beta_2 (1), \\
        & \beta_1 (\infty) < \alpha(\infty) < \beta_2 (\infty).
    \end{align*}
    Once again $Q = 3 - 2 = 1$, and so the inequality for $\cal L \oplus \theta(\cal L) (-Q)$ is 
    $$
    1 + \alpha(0) + \alpha(1) + \alpha(\infty) + \beta_1 (0) + \beta_1 (1) + \beta_1 (\infty) \leq 0.
    $$
    Now we need to apply the shift operation at each point of $S$. The difference here is that we have that each flag of $\cal V$ must come from $\cal L$ in this case, and so this example highlights shifting to nongeneric bundles.

    As before we can suppose that we are in the situation of $\cal L = \cal O(1) \to \cal W = \cal O(1) \oplus \cal O(1)$. If we then apply the shifting operation at each point $p \in S$, we get that
    $$
    \tilde{\cal L} = \cal O(4) \to \tilde{\cal W} = \calO(1) \oplus \calO(4).
    $$
    And so here we have $\tilde{\cal V} \simeq \calO \oplus \calO(3).$ Since $\cal V = \cal O \oplus \cal O$, we can again take $\delta = 0$ in this case, so we are looking at $\cal S = \cal O \hookrightarrow \cal V = \cal O \oplus \cal O.$ We then need $\cal O (N)$ to be a valid subbundle of $\cal O \oplus \cal O(3) \simeq \tilde{\cal V}$. In this case, we can take $N = 3$ for the worst case, i.e. we can have our subbundle $\cal S$ hit the special flag of $\cal V_p$ for every $p \in S$. Because this situation is maximal, we get just one semistability inequality:
    $$
    \beta_2 (0) + \beta_2 (1) + \beta_2 (\infty) \leq 0.
    $$
\end{ex}

\subsection{Possible connection to tropical geometry}
The problem of enumerating stability conditions for $(1,2)$ systems Hodge bundles with non-generic rank 2 component is essentially the same as the problem of enumerating subbundles of the non-generic rank 2. In such a case, suppose for simplicity that we are left with the situation that $\theta : \cal L \simeq \cal O \to \cal O \oplus \cal O(-a)$. Consider then the Hirzebruch surface $\PP (\cal O \oplus \cal O(-a))$, and consider a section $C \subset \PP (\cal O \oplus \cal O(-a)))$. By universal property characterizing projective bundles, we must have that the section $C \simeq \PP^1$ corresponds to a line subbundle of $\cal O \oplus \cal O(-a)$. The degree of this subbundle is then determined by the divisor class of $C$ inside $\PP(\cal O \oplus \cal O(-a))$. Suppose we set $F$ to be the fiber class and $C_0$ the 0-section class, then all sections of $\pi : \PP(\cal O \oplus \cal O(-a)) \to \PP^1$ must have divisor class of the form $C = C_0 + b F$. Then to understand the stability conditions for $\cal O \oplus \cal O(-a)$, we are left with understanding the curves corresponding to certain section classes living inside the Hirzebruch surface $\PP( \cal O \oplus \cal O(-a))$. 

The $(1,2)$ system of Hodge bundles situation prescribes that we have certain flags, some coming from $\theta : \cal O \to \cal O \oplus \cal O(-a)$, which would correspond to points lying on the infinity section of $\PP(\cal O \oplus \calO(-a))$, and other flags coming without conditions, which would become generic points on the Hirzebruch surface. The question is then to compute the (logarithmic) Gromov-Witten invariants of curves passing through certain prescribed generic points along with adding certain incidence conditions with respect to the infinity section. Such counts have been computed using techniques from tropical geometry as investigated in \cite{lian2025enumeratinglogrationalcurves}.

\section{The \texorpdfstring{$(1,1)$}{(1,1)} case}
\label{section: (1,1)}
If $(\cal E, \theta)$ is our system of Hodge bundles, then being in the $(1,1)$ case means that $\cal E = \cal L \oplus \cal L'$ with $\rank \cal L = \rank \cal L' = 1$, and $\theta (\cal L) \subset \cal L' \otimes \Omega^1 (\log S)$. Since $\cal L$ and $\cal L'$ are rank 1, we only have one subspace for each component and hence one weight for each component, which we will denote as $\alpha(p)$ and $\alpha'(p)$ for each $p \in S$ respectively. Furthermore, for simplicity, we will restrict to $\pardeg \cal E = 0$. 

To check for stability conditions, it suffices to only look at $\theta$-stable subsheaves compatible with the Hodge decomposition, so, in this case, the only such bundle we need to look at is $\cal L'$. Since we are only looking at stable parabolic Higgs bundles of degree zero, we immediately get the inequality
$$
\deg \cal L' + \sum_{p \in S} \alpha'(p) < \deg \cal L + \deg \cal L' + \sum_{p \in S} (\alpha(p) + \alpha'(p)) = 0.
$$

Further, we will need to control the poles of the map $\theta$. To do this, we look at 
$$
\Res_p \theta : \cal L_p \to \cal L'_p.
$$
This breaks down into two cases: we have that $\alpha'(p) \leq \alpha(p)$ or $\alpha'(p) > \alpha(p)$. In the case that $\alpha'(p) \leq \alpha(p)$, then since $\Res_p \theta$ takes $\cal L_{\alpha(p)} \to \cal L'_{\hat \alpha(p)}$ and $\alpha(p) \geq \alpha'(p)$ and hence $\hat \alpha(p) \geq \alpha' (p)$, this means that $\Res_p \theta$ sends $\cal L_p$ to 0. Hence we do not have a pole at $p \in S$. In the case that $\alpha'(p) > \alpha(p)$, this means that $\alpha'(p) = \hat \alpha(p)$, and $\res_p \theta$ sends $\cal L_p = \cal L_{\alpha(p)}$ to $\cal L'_p = \cal L'_{\alpha'(p)}$, and it cannot be zero, so $\theta$ must have a pole here. Define for $p \in S$
$$
\ell(p) = \begin{cases}
    1 & \text{if } \alpha'(p) > \alpha(p), \\
    0 & \text{else, i.e. } \alpha'(p) \leq \alpha(p).
\end{cases}
$$
$\ell(p)$ counts the number of poles of $\theta$ at $p$, and hence $\sum_{p \in S} \ell(p)$ counts the number of poles of $\theta$. This means then that
$$
\theta : \cal L \to \cal L' \otimes \Omega^1 \paren*{\sum_{p \in S} \ell (p)},
$$
which gives us the degree inequality
$$
\deg \cal L \leq \deg \cal L' + \paren*{-2 + \sum_{p \in S} \ell(p)}.
$$

\subsection{Examples with three points and generic bundle}
Suppose $S = \set{0, 1, \infty}$, all the weights of $\cal E$ are normalized to be between 0 and 1. We also assume that our bundle $\cal E$ is evenly split, i.e. $\cal E \simeq \cal O(a) \oplus \cal O(b)$ where $|a-b| \leq 1$. In this situation, $\sum_{p \in S} (\alpha(p) + \alpha'(p))$ is either even or odd. In either case, there are actually only a finite number of situations that can arise. We will examine each situation here.

Suppose that $\sum_{p \in S} (\alpha(p) + \alpha'(p)) = 2N + 1$ for some integer $N$. Since our bundle has generic splitting type, $\cal E \simeq \calO(-N) \oplus \calO(-N-1)$. So then $\cal L$ and $\cal L'$ could each be either one of $\calO(-N)$ or $\calO(-N-1)$.

Suppose that $\cal L' = \calO(-N-1)$ and $\cal L = \calO(-N)$. So then we have
$$
\theta : \calO(-N) \to \calO(-N-1) \otimes \Omega^1(\sum_{p \in S} \ell(s)).
$$
This condition requires us to have that $\deg \calO(-N) \leq \deg \calO(-N-1) \otimes \Omega (\sum_{p \in S} \ell(p))$, i.e.
$$
-N \leq -N - 3 + \sum_{p \in S} \ell(p),
$$
which implies that the number of poles of $\theta$ must be exactly 3. So then in this case, we have that $\alpha'(p) > \alpha(p)$ for every $p \in \set{0,1,\infty}$. Next, note that the stability condition means that
$$
- N - 1 + \sum_{p \in S} \alpha' (p) < -2N - 1 + \sum_{p \in S} (\alpha(p) + \alpha'(p)).
$$
Rearranging, we see that $N < \sum_{p \in S} \alpha(p) \leq 3$. Therefore, we must then have that $0 \leq N < 3$. So $N = 0, 1, 2$. The question now is then that for each such $N$, do we have a stable subbundle?
\begin{itemize}
    \item $N = 0$. Here $\cal L = \calO$, $\cal L' = \calO(-1)$, $\cal E = \calO \oplus \calO(-1)$. The number of poles is 3, so we need that every $\alpha'(p) > \alpha(p)$, and $\sum_{p \in S} \alpha(p) + \alpha'(p) = 1$. Try each $\alpha'(p) = \frac{2}{9}$ and $\alpha(p) = \frac{1}{9}$. From here, checking the inequality we have
    $$
    -1 + 3 \paren*{\frac{2}{9}} = 1 + \frac{2}{3} = -1/3 < 0. 
    $$
    So this example works. Which gives us existence in this case.

    \item $N = 1$. Here $2N + 1 = 3$, and our bundle $\cal E = \calO(-1) \oplus \calO(-2)$. Consider the following weights.
    \begin{center}
        \begin{tabular}{c c}
            $\alpha'(p)$ & $\alpha(p)$ \\
            \hline
            $7/9$ & $6/9$ \\
            $2/9$ & $1/9$ \\
            $6/9$ & $5/9$ 
        \end{tabular}    
    \end{center}
    Checking stability, we check
    $$
    -2 + \frac{7}{9} + \frac{2}{9} + \frac{6}{9} = -2 + \frac{15}{9} = -3/9 < 0
    $$
    as desired.

    \item $N = 2$. Here $2N + 1 = 5$, and $\cal E = \calO(-2) \oplus \calO(-3)$. $\cal L = \calO(-2)$, $\cal L' = \calO(-3)$. We need to check that $-3 + \sum_{p \in S} \alpha'(p) < 0$. This can be achieve by the following system of weights.
    \begin{center}
        \begin{tabular}{c c}
            $\alpha'(x)$ & $\alpha(x)$ \\
            \hline
            $14/15$ & $13/15$ \\
            $14/15$ & $2/15$ \\
            $14/15$ & $2/15$
        \end{tabular}
    \end{center}
    Checking stability we have
    $$
    -3 + 3 \paren*{\frac{14}{15}} = -\frac{3}{15} < 0.
    $$
\end{itemize}

Next, we try the case $\cal L = \calO(-N-1)$ and $\cal L' = \calO(-N)$. We again have that
$$
\theta : \calO(-N-1) \to \calO(-N) \otimes \Omega^1(\log S),
$$
so
$$
-N - 1 \leq -N - 2 + \sum_{p \in S} \ell(p),
$$
which implies that the number of poles must be greater than or equal to 1. Checking for permissible $N$, we have
$$
- N + \sum_{p \in S} \alpha'(p) < -2N - 1 + \sum_{p \in S} \alpha(p) + \alpha'(p),
$$
so rearranging we get $N < -1 + \sum \alpha(p) \leq 2$, so $N = 0, 1$.

In the $N = 0$ case we need that 
$$
0 + \sum_{p \in S} \alpha'(p) < 0
$$
which is impossible because $\alpha'(p) \geq 0$. So there are no stable parabolic Higgs bundles in this case.

So we are left with the case $N = 1$. Suppose we have 3 poles, then we get that $\alpha' (p) > \alpha(p)$ for every $p \in S$, and $3 = 2N + 1 = \sum_{p \in S} \alpha(p) + \alpha'(p) \leq 2 \sum_{p \in S} \alpha'(p)$, which implies that 
$$
\frac{3}{2} \leq \sum_{p \in S} \alpha'(p),
$$
so in fact there are no stable bundles in this case as well because stability requires us to have
$$
-1 + \sum_{p \in S} \alpha'(xp) < 0,
$$
which is not true for $\sum_{p \in S} \alpha'(p) > 3/2$. So we need only to find the cases of 2 poles and 1 pole. 

In the case of two poles, there is one $p \in S$ where $\theta$ is holomorphic. We can assume WLOG that $p = 0$, in which case we must have that $\alpha'(0) \leq \alpha(0)$, $\alpha'(1) > \alpha(1)$ and $\alpha'(\infty) > \alpha(\infty)$. Stability is equivalent to 
$$
\alpha'(0) + \alpha'(1) + \alpha'(\infty) < 1.
$$
Now since $\alpha(1) + \alpha(\infty) < \alpha'(1) + \alpha'(\infty) \leq \sum_{p \in S} \alpha'(p) < 1$, we necessarily have that
$$
3 = \sum_{p \in S} \alpha(p) + \alpha'(p) < 2 + \alpha(0),
$$
which then implies that $\alpha(0) > 1$, which is impossible.

So we are finally left with the case of only one pole. In this case, the inequalities could be satisfied with the following system of weights.
\begin{center}
    \begin{tabular}{c c}
        $\alpha'(x)$ & $\alpha(x)$ \\
        \hline
        $1/9$ & $8/9$ \\
        $1/9$ & $6/9$ \\
        $6/9$ & $5/9$
    \end{tabular}
\end{center}

Next, we deal with the case $\sum_{p \in S} \alpha(p) + \alpha'(p) = 2N$. So here the generic $\cal E = \calO(-N) \oplus \calO(-N)$, and 
$$
\theta : \calO(-N) \to \calO(-N) \otimes \Omega^1 \paren*{\sum_{p \in S} \ell(p)},
$$
which gives us that
$$
-N \geq -N -2 + \#\text{poles}
$$
so we have that the number of poles is greater than or equal to 2. We can then further control what $N$'s are allowed by checking 
$$
-N + \sum_{p \in S} \alpha'(p) < -2N + \sum_{p \in S} \alpha(p) + \alpha'(p),
$$
and then rearranging to see that $N < \sum_{p \in S} \alpha(p) \leq 3$, which again implies that $N = 0, 1, 2$.

\begin{itemize}
    \item $N=0$. Again this gives us $\sum_{p \in S} \alpha'(p) < 0$, which is impossible.

    \item $N = 1$. In this case, we have that $\sum_{p \in S} \alpha(p) + \alpha'(p) = 2$ and stability means
    $$
    \sum \alpha'(p) < 1.
    $$
        
    When we have two poles, we can assume WLOG that $\theta$ is holomorphic at say $0$. This means that we need the weights to satisfy $\alpha'(0) \leq \alpha(0), \alpha'(1) > \alpha(1), \alpha'(\infty) > \alpha(\infty).$
    In this case, an example would be the following data.
    \begin{center}
        \begin{tabular}{c c}
            $\alpha'(p)$ & $\alpha(p)$ \\
            \hline
            $1/6$ & $5/6$ \\
            $2/6$ & $1/6$ \\
            $2/6$ & $1/6$
        \end{tabular}
    \end{center}

    When we have three poles, then this means that $\alpha'(0) > \alpha(0), \alpha'(1) > \alpha(1), \alpha'(\infty) > \alpha(\infty)$. From this, we get that
    $$
    2 = \sum (\alpha' + \alpha) < 2 \sum \alpha'.
    $$
    From the stability condition that $\sum \alpha' < 1$, we get then that 
    $$
    2 < 2 \sum \alpha' < 2,
    $$
    which is impossible.

    \item $N=2$. Here we have $\sum \alpha' + \alpha = 4$ and the stability condition is $\sum \alpha' < 2$.

    Again, the two pole situation is when $\alpha'(0) \leq \alpha(0), \alpha'(1) > \alpha(1), \alpha'(\infty) > \alpha(\infty)$. An example of admissible weight data is the following.
    \begin{center}
        \begin{tabular}{c c}
            $\alpha'(p)$ & $\alpha(p)$ \\
            \hline
            $1/6$ & $5/6$ \\
            $5/6$ & $4/6$ \\
            $5/6$ & $4/6$
        \end{tabular}
    \end{center}

    Finally when we have 3 poles, then we need $\alpha'(p) > \alpha(p)$ for every $p \in S$. Again we can check that we cannot get a stable bundle because 
    $$
    4 = \sum \alpha' + \alpha < 2 \sum \alpha',
    $$
    and $\sum \alpha' < 2$ from the stability condition, so we get $4 < 4$, impossible. 
\end{itemize}

\subsection{General condition for existence}
The above analysis of the generic bundle case and three points give good indication for what happens in the general case. Again let $\cal E = \cal L \oplus \cal L'$ be our bundle. Suppose the weights of $\cal L$ are $\alpha(p)$ and the weights of $\cal L'$ are $\alpha'(p)$ for $p \in S$, for $S$ of any size. Assume also that the weights are normalized so that they all fall between 0 and 1.

\begin{prop}
\label{parstab cond}
    A stable parabolic system of Hodge bundles over $\PP^1$ with parabolic structure over $S = \set{p_1, \ldots, p_n}$ of type $(1,1)$ with semisimple local monodromies exists if and only if there is a system of weights $(\alpha(p), \alpha'(p))_{p \in S}$ such that
    \begin{itemize}
        \item If $\sum_{p \in S} (\alpha (p) + \alpha'(p)) = 2N + 1$, then 
        $$
        \frac{3}{2} - \frac{1}{2} \sum_{p \in S} \ell(p) < N + 1 - \sum_{p \in S} \alpha'(p);
        $$
        \item If $\sum_{p \in S} (\alpha(p) + \alpha'(p)) = 2N$, then
        $$
        1 - \frac{1}{2} \sum_{p \in S} \ell(p) < N - \sum_{p \in S} \alpha'(p).
        $$
    \end{itemize}
\end{prop}
\begin{proof}
We tackle first the case $\sum_{p \in S} (\alpha(x) + \alpha'(x)) = 2N$. Then for some $k \in \Z$, our bundle must be $\calO(-N-k) \oplus \calO(-N + k),$ where now we can assume $\cal L \simeq \calO(-N -k)$ and $\cal L' \simeq \cal O(-N + k)$. Let 
$$
\ell(p) = \begin{cases}
    1 & \text{if $\alpha'(p) > \alpha(p)$,} \\
    0 & \text{else.}
\end{cases}
$$
We have then that $\theta : \cal L \to \cal L' \otimes \Omega^1 (\log S)$ becomes
$$
\theta : \calO(-N - k) \to \calO(-N + k) \otimes \calO \paren*{-2 + \sum_{p \in S} \ell(p)},
$$
which tells us that $-N - k \leq -N + k -2 + \sum_{p \in S} \ell(p)$, i.e. the 
$$
k \geq 1 - \frac{1}{2} \sum_{p \in S} \ell(p).
$$
And then from the stability condition we get that
$$
\deg \cal L' + \sum_{p \in S} \alpha' (p) < 0.
$$
Then since $\cal L' = \calO(-N + k)$, this condition is exactly
$$
-N + k  + \sum_{p \in S} \alpha'(p) < 0.
$$
So then we can bound $k$ to be 
$$
1 - \frac{1}{2} \sum_{p \in S} \ell(p) \leq k < N - \sum_{p \in S} \alpha'(p).
$$

In the case of $\sum \alpha + \alpha' = 2N + 1$, we have something similar:
$$
\cal E = \calO(-N - k) \oplus \calO(-N - 1 + k),
$$
where $\cal L \simeq \calO(-N-k)$ and $\cal L' \simeq \calO(-N-1+k)$. With $\ell(p)$ defined the same as before, 
$$
-N - k \leq -N - 1 + k - 2 + \sum_{p \in S} \ell(p),
$$
so we have
$$
k \geq \frac{3}{2} - \frac{1}{2} \sum_{p \in S} \ell(p).
$$
The stability condition finally becomes
$$
-N - 1 + k + \sum_{p \in S} \alpha'(p) < 0.
$$
So we get
$$
\frac{3}{2} - \frac{1}{2} \sum_{p \in S} \ell(p) \leq k < N + 1 - \sum_{p \in S} \alpha'(p).
$$

Of course, any data that satisfies the conditions specified can be a valid parabolic system of Hodge bundles. 
\end{proof}

\subsection{Comparison with regular parabolic stability}
Note that for regular parabolic stability, we need that every subbundle with the induced filtration must have parabolic degree less than zero. We will see that this is explicitly not the case in our Higgs bundle setting. 

\begin{prop}
    Stable parabolic system of Hodge bundles of type $(1,1)$ are not stable as regular parabolic bundles. 
\end{prop}

\begin{proof}
    Suppose after applying shift operations we have that $\sum_{p \in S} \alpha(p) + \alpha'(p) = 0$. This means that $N = 0$ in Proposition \ref{parstab cond}, so our bundle must look like
    $$
    \calO(-k) \oplus \calO(k),
    $$
    where $\theta : \calO(-k) \to \calO(k) \otimes \Omega^1 (\log S).$ The parabolic degree 0 condition becomes
    $$
    \paren*{-k + \sum_{p \in S} \alpha(p)} + \paren*{k + \sum_{p \in S} \alpha'(p)} = 0,
    $$
    and the parabolic stability condition tells us that 
    $$
    k + \sum_{p \in S} \alpha'(p) < 0.
    $$
    But then this implies that 
    $$
    - k + \sum_{p \in S} \alpha(p) > 0,
    $$
    which contradicts stability, since this is the parabolic degree of the $\calO(-k)$ subbundle.
\end{proof}

\subsection{Comparison with Biswas' criterion}
Biswas provided a criterion for the existence of rank 2 stable parabolic bundles of degree zero with prescribed weights. Suppose that after applying our shift operations on our bundles, we have that the parabolic weights sum to zero. Then in this setting, Biswas' criterion becomes the following theorem.
\begin{thm*}[\cite{Biswas1998ACF}]
\label{higgs unstable}
    There is a parabolic bundle of degree zero and with parabolic weights $\set{\alpha(p) \geq \beta(p)}$ at $p \in S$ if and only if for any subset $A \subset S$ of cardinality $2j + 1$, with $j$ a non-negative integer, the following inequality holds:
    $$
    -j + \sum_{p \in A} \alpha(p) + \sum_{p \in S \sm A} \beta(p) < 0.
    $$
\end{thm*}

Now in the case of $n = |S| = 3$, we have that our local system is necessarily rigid in the sense of \cite{katz1996rigid}. So if $\cal E = \cal L \oplus \cal L'$ is our parabolic Higgs bundle, then any other local system with local monodromies conjugate to the ones given by the Higgs bundle must give an isomorphic local system. But now, Biswas' criterion gives a necessary and sufficient condition for the existence of a $U(2)$ local system, whereas our conditions guarantee the existence of $U(1,1)$ local system. In particular, because we require our local systems to be irreducible, this necessarily means that we cannot satisfy both Biswas' criterion and the criterion of Proposition \ref{parstab cond}. We will show this directly:
\begin{prop}
    Suppose $|S| = 3$ and $\set{\alpha(p), \alpha'(p)}_{p \in S}$ is a system of weights that satisfy Biswas' criterion. Suppose further that $\alpha(p) + \alpha'(p) = 0$ for all $p \in S$, i.e. we have local monodromies of determinant one. Then this system of weights cannot satisfy the criterion given in Proposition \ref{parstab cond}.
\end{prop}

\begin{proof}
    To fix notation, suppose $S = \set{0, 1, \infty}$ and set $\ell = \sum_{p \in S} \ell(p)$ for ease of notation. $\ell$ could be $0, 1, 2,$ or $3$. We will deal with ease of these cases separately. Now from Proposition \ref{parstab cond}, we see that 
    $$
    1 - \frac{1}{2} \ell \leq k < - \sum_p \alpha'(p) = \sum_p \alpha(p),
    $$
    and recall that we are considering the stability of the bundle $\cal L \oplus \cal L'$ with Higgs field $\theta$ sending $\cal L$ to $\cal L'$. We have in this case that $\cal L \simeq \calO (-k)$ and $\cal L' \simeq \calO(k)$. We further have that $\max (\alpha(p), \alpha'(p)) \leq 1/2$ because $\alpha(p) + \alpha'(p) = 0$, and $|\alpha(p) - \alpha'(p)| \leq 1$ because the shift operation preserves their distances.

    Now we consider each case of $\ell$.
    \begin{enumerate}[label=]
        \item $\ell = 0$: In this case we have that 
        $$
        1 \leq k < \sum_p \alpha(p) \leq 3/2,
        $$
        so the only possible $k$ is $1.$ In this case, set $j = 1$, so then $A \subset S$ has cardinality $2 + 1 =3$, so $A = S$. Now since $\ell = 0$, we have that each $\alpha'(p) \leq \alpha(p)$, so then Biswas' inequality becomes
        $$
        -1 + \alpha(0) + \alpha(1) + \alpha(\infty) < 0.
        $$
        But, this implies that $\cal L$ satisfies the stability condition, i.e. we have that $\pardeg \cal L < 0 = \pardeg \cal L \oplus \cal L'$, but this necessarily implies that $\pardeg \cal L' > 0$ as $\pardeg \cal L + \pardeg \cal L' = 0$, which contradicts the stability condition for $\cal L \oplus \cal L'$ to be a stable parabolic system of Hodge bundles.

        \item $\ell = 1$: In this case, suppose WLOG that $\alpha'(0) > \alpha(0)$. This then implies that $\alpha(0) < 0$, which then implies that
        $$
        \frac{1}{2} \leq k < \sum \alpha(p) \leq 1.
        $$
        So in this case, no such $k$ is allowed.

        \item $\ell = 2$: Suppose WLOG that $\alpha'(0) > \alpha(0)$ and $\alpha'(1) > \alpha(1)$, so we have that $\alpha(0), \alpha(1) < 0$. So then we have that
        $$
        0 \leq k < 1/2.
        $$
        So the only possible $k$ is 0. Now use the Biswas criterion for $j = 0$. So in this case $|A| = 1$. In this case, let $A = \set{\infty}$, so then Biswas' inequality becomes
        $$
        \alpha(\infty) + \alpha(0) + \alpha(1) < 0.
        $$
        So again, this contradicts parabolic Higgs stability.

        \item $\ell = 3$: This is the case where each $\alpha(p) \leq \alpha'(p)$. But then this implies that
        $$
        - \frac{1}{2} \leq k < 0,
        $$
        and no such $k$ can satisfy this condition.
    \end{enumerate}
\end{proof}

\section{The \texorpdfstring{$(1,1,\ldots, 1)$}{(1,1,...,1)} case}
\label{section: (1,...,1)}
We can further examining the case of $(1,1,\ldots, 1)$ systems of Hodge bundles. In this situation, there are no particular subtleties involved with subbundles and flags, and so stability in this case is a purely numerical phenomenon.

In this setting, we suppose our bundle $\cal E = \cal E_1 \oplus \ldots \oplus \cal E_N$ with each $\cal E_i$ being a line bundle. Our Higgs field is $\theta : \cal E \to \cal E \otimes \Omega^1 (\log S)$ with $\theta_i := \theta \mid_{\cal E_i} : \cal E_i \to \cal E_{i+1} \otimes \Omega^1 (\log S)$. Denote the weights of $\cal E_i$ at $p \in S$ by $\alpha_i (p)$. Suppose again that we have semisimple monodromies on the level of local systems, then this implies that we have a strongly parabolic bundle and in particular we have that
$$
\Res_p \theta (\cal E_{i,\alpha}) \subset \cal E_{i - 1,\hat{\alpha}}, 
$$
where $\hat{\alpha}$ is the next highest weight after $\alpha$. 

If we have that $\alpha_i (p) \leq \alpha_{i-1} (p)$, then since 
$$
\Res_p \theta (\cal E_{i-1,\alpha_i (p)}) \subset \cal E_{i, \hat{\alpha_i (p)}} = 0,
$$
this means that $\theta_i : \cal E_i \to \cal E_{i+1} \otimes \Omega^1 (\log S)$, has no pole at $p$, and so we can regard this map actually as a morphism $\cal E_i \to \cal E_{i+1} \otimes \Omega( \log (S \sm \set{p}) )$. We can use this information to restrict the possible degrees of $\cal E_i$.

Similar to the $(1,1)$ case, define
$$
\ell_i (p) = \begin{cases}
    1 & \text{if $\alpha_i(p) > \alpha_{i-1} (p)$}, \\
    0 & \text{if $\alpha_i (p) \leq \alpha_{i-1}(p)$.}
\end{cases}
$$
So then from here, we see that $\ell_i := \sum_{p \in S} \ell_i (p)$ counts the number of poles of $\theta_i$. And since $\theta_i$ is a nonzero map
$$
\theta_i : \cal E_i \to \cal E_{i+1} \otimes \Omega^1 \paren*{\sum_{p \in S} \ell_i (p)},
$$
we must have then that $d_i := \deg \cal E_i \leq d_{i+1} + \paren*{-2 + \sum_{p \in S} \ell_i (p)}$. 

In the $(1,1,\ldots, 1)$ case, the only $\theta$-stable subbundles compatible with the Hodge decomposition is $\cal E_m \oplus \ldots \oplus \cal E_N$, so the semistability condition tells us that
$$
\sum_{i=m}^N d_i + \sum_{i=m}^N \sum_{p \in S} \alpha_i (p) \leq 0,
$$
with $d_i \leq d_{i+1} -2 + \sum_{p \in S} \ell_i (p)$.

In the case where our bundle is of type $(1,1,1)$, we can combine all the inequalities into one single long string of inequalities. In this case, we will denote our parabolic system of Hodge bundles as $\cal E \simeq \cal L_1 \oplus \cal L_2 \oplus \cal L_3$ with $\theta (\cal L_i) \subset \cal L_{i+1} \otimes \Omega^1 (\log S)$. Furthermore, we can also repeatedly apply the parabolic shifting operation described in Section \ref{section:shifting} so that the sum of weights is equal to 0, i.e. 
$$
\sum_{p \in S} \sum_{i=1}^3 \alpha_i (p) = 0.
$$

\begin{prop}
    Let $(\alpha_1(p), \alpha_2 (p), \alpha_3 (p))_{p \in S}$ be a system of numbers such that
    $$
    \sum_{i=1}^3 \sum_{p \in S} \alpha_i (p) = 0
    $$
    and let $(d_1, d_2, d_3) \in \Z^3$. 
    
    There exists a semistable type $(1,1,1)$ parabolic system of Hodge bundles $(\cal E, \theta)$ of parabolic degree 0 with $\cal L_i$ having weights $\alpha_i(p)$ and $\deg \cal L_i = d_i$ if and only if the following numerical conditions are satisfied:
    \begin{align*}
        \bullet \quad & \sum_{p \in S} \alpha_2 (p) + \alpha_3 (p) \leq d_1 \leq d_2 - 2 + \ell_1 \leq d_3 - 4 + \ell_1 + \ell_2 \leq -\sum_{p \in S} \alpha_3 (p) - 4 + \ell_1 + \ell_2, \\
        \bullet \quad & d_1 + d_2 + d_3 = 0.
    \end{align*}
\end{prop}

\begin{proof}
We assumed that $\pardeg \cal E = 0$, so we have that
$$
\pardeg \cal E = d_1 + d_2 + d_3 + \sum_{i=1}^3 \sum_{p \in S} \alpha_i (p) = 0.
$$
But we have that $\sum_{i=1}^3 \sum_{p \in S} \alpha_i (p) = 0$, which implies that $d_1 + d_2 + d_3 = 0$.

The only $\theta$-stable subsheaves to check in this situation are $\cal L_3$ and $\cal L_2 \oplus \cal L_3$. The semistability inequality on $\cal L_3$ gives us
$$
d_3 + \sum_{p \in S} \alpha_3 (p) \leq 0,
$$
which gives us that 
$$
d_3 \leq - \sum_{p \in S} \alpha_3 (p).
$$
The semistability condition on $\cal L_2 \oplus \cal L_3$ gives
$$
d_2 + d_3 + \sum_{p \in S} \alpha_2 (p) + \alpha_3 (p) \leq 0.
$$
Since $d_1 + d_2 + d_3 = 0$, we have that $d_1 = -d_2 - d_3$, and hence we get
$$
\sum_{p \in S} \alpha_2 (p) + \alpha_3 (p) \leq d_1,
$$
which is the first part of our inequalities. Now since $\cal L_1 \to \cal L_2 \otimes \Omega^1 (\log S)$ is a nonzero map, we have that
$$
d_1 \leq d_2 - 2 + \sum_{p \in S} \ell_1 (p),
$$
and similarly since $\cal L_2 \to \cal L_3 \otimes \Omega (\log S)$, we get that $d_2 \leq d_3 - 2 + \sum_p \ell_2 (p)$. Putting everything together, we get that
\begin{align*}
\sum_{p \in S} \alpha_2 (p)  + \alpha_3 (p) &\leq d_1 \leq d_2 - 2 + \sum_{p \in S} \ell_1 (p) \leq d_3 - 4 + \sum_{p \in S} \ell_1 (p) + \ell_2 (p) \leq \\
&- \sum_{p \in S} \alpha_3 (p) -4 + \sum_{p \in S} \ell_1 (p) + \ell_2 (p).
\end{align*}
Replacing $\sum_{p \in S} \ell_i(p)$ by $\ell_i$ gives us the system of inequalities that we desire.

Conversely given numerical data that satisfies inequalities given, it is easy to construct a semistable parabolic system of Hodge bundles of type $(1,1,1)$ with these numerical conditions. 
\end{proof}

\begin{rmk}
    Since the conditions for the $(1,1,\ldots, 1)$ case are purely numerical and require no subbundle enumeration, we can easily change the inequalities above to strict inequalities and get necessary stability inequalities.
\end{rmk}

\bibliography{ref}
\bibliographystyle{alpha}

\end{document}